\documentclass[11pt]{article}
\usepackage{amsmath,enumerate,amsfonts,color,amssymb,amsthm}

\usepackage{tikz}

\def\RR{{\mathbb R}}

\def\mcB{{\mycal B}}

\def\mcF{{\mycal F}}

\def\mcP{{\mycal P}}

\def\eps{{\varepsilon}}

\newtheorem{theorem} {\sc  Theorem\rm} [section]

\newtheorem{lemma} [theorem] {\sc  Lemma\rm}

\newtheorem{prop} [theorem] {\sc  Proposition\rm}

\newtheorem{remark}[theorem]{\sc  Remark\rm}

\newcounter{marnote}

\DeclareFontFamily{OT1}{rsfs}{}
\DeclareFontShape{OT1}{rsfs}{m}{n}{ <-7> rsfs5 <7-10> rsfs7 <10-> rsfs10}{}
\DeclareMathAlphabet{\mycal}{OT1}{rsfs}{m}{n}

\def\dist{{\rm dist}}
\def\tr{{\rm tr}}

\def\mcS{{\mycal{S}}}

\def\be{\begin{equation}}
\def\ee{\end{equation}}

\def\Id{{\rm Id}}
\def\dist{{\rm dist}}
\def\tr{{\rm tr}}

\def\mcS{{\mycal{S}}}

\newcommand{\R}{\mathbb{R}}
\newcommand{\Z}{\mathbb{Z}}
\newcommand{\N}{\mathbb{N}}

\def\be{\begin{equation}}
\def\ee{\end{equation}}
\def\bea#1\eea{\begin{align}#1\end{align}}
\def\non{\nonumber}

\numberwithin{equation}{section}

\usepackage[english]{babel}
\usepackage{enumitem,esint}

\renewcommand{\d}{\mathrm{d}}
\renewcommand{\O}{\mathrm{O}}
\newcommand{\SO}{\mathrm{SO}}
\newcommand{\abs}[1]{\left| #1 \right|} 
\newcommand{\norm}[1]{\left\| #1 \right\|}
\newcommand{\csubset}{\subset\!\subset}
\newcommand{\mres}
{\mathbin{\vrule height 1.6ex depth 0pt width 0.13ex\vrule height 0.13ex depth 0pt width 1.3ex}}

\begin{document}
\title{Design of effective bulk potentials\\ for nematic liquid crystals\\ via colloidal homogenisation }

\author{ Giacomo Canevari\thanks{BCAM,  Basque  Center  for  Applied  Mathematics,  Mazarredo  14,  E48009  Bilbao,  Bizkaia,  Spain.
(gcanevari@bcamath.org)}\, and Arghir Zarnescu\thanks{IKERBASQUE, Basque Foundation for Science, Maria Diaz de Haro 3,
48013, Bilbao, Bizkaia, Spain.}\,  \thanks{BCAM,  Basque  Center  for  Applied  Mathematics,  Mazarredo  14,  E48009  Bilbao,  Bizkaia,  Spain.
}\, \thanks{``Simion Stoilow" Institute of the Romanian Academy, 21 Calea Grivi\c{t}ei, 010702 Bucharest, Romania.}} 

\maketitle

\begin{abstract} 
 We consider a Landau-de Gennes model for a suspension of 
 small colloidal inclusions in a nematic host. We impose suitable anchoring
 conditions at the boundary of the inclusions, and we work in the 
 dilute regime --- i.e., the size of the inclusions is much smaller 
 than the typical separation distance between them, so that the total
 volume occupied by the inclusions is small.
 By studying the homogenised limit, and proving rigorous convergence 
 results for local minimisers, we compute the effective free energy 
 for the doped material. In particular, we show that not only 
 the phase transition temperature, but any coefficient of the quartic 
 Landau-de Gennes bulk potential can be tuned,
 by suitably choosing the surface anchoring energy density.
\end{abstract}




\section{Introduction}

Colloids are mixtures in which microscopic-size particles are suspended in an ambient fluid. The host fluid can be a standard, Newtonian, fluid, or a complex fluid, for instance a liquid crystal. Colloids provide an impressive and fascinating number of applications, offering the possibility to use directed self-assembly to realize unusual composite materials with apriori given properties, designable  photonic crystals and nanoparticles transport, to count just a number of recent, exciting applications. 

Some of the most striking applications are those in which the ambient fluid is a liquid crystal, because this allows the colloidal particles to take advantage of the unusual long-range elastic properties of the liquid crystals. The field of liquid crystal colloids is a fast emerging area of research in condensed matter physics, despite being only about two decades old, see \cite{Lavrentovich,Smalyukh}. The mathematical studies are  few \cite{ABL1,ABL2, Zhang,BCD, ferromagnetic, multiscale} and the most recent ones are generally focused on the defect patterns generated by the presence of the colloids. 

Our study in this article is focused on the bulk properties generated in the ambient material  by the presence of the colloidal inclusions. Related mathematical studies in this directions are those of \cite{BCD, ferromagnetic, multiscale}. The mentioned works show that in a dilute  regime the  mixture of colloids and liquid crystal behave like a {\it ``homogeneous"}, new material, with properties that can be  quantitatively parametrised. On the other hand, let us note that  in the physics literature there are a number of works \cite{Reznikov,Fenghua} showing that in such a dilute regime the mixture behaves like a {\it pure} liquid crystal, but with {\it enhanced properties}. The aim of this article is to recover in the simplest mathematical setting this observation from the physics literature, namely to determine under what conditions the homogenised colloidal material can behave as an {\it improved} standard liquid crystal material. We will also study this issue from a design perspective, trying to determine what kind of colloidal particles are necessary in order to obtain from a given ambient liquid crystal material an apriori prescribed enhanced liquid crystal. 

\bigskip
Our matematical setting will be the the simplest Landau-de Gennes variational theory of liquid crystals.  We denote $$\mcS_0:=\{ Q\in \RR^{3\times 3},\, Q=Q^{\mathsf{T}},\,\tr (Q)=0\}$$ and refer to its elements as {\bf $Q$-tensors}. The elements in this set describe the characteristic feature of the material, namely the local orientational ordering  (for more details see \cite{Mottram, BallZarnescu}).  We take   $\Omega\subseteq\R^3$ to be a bounded,  Lipschitz domain that models the ambient fluid, and let~$g\in H^{1/2}(\partial\Omega, \, \mcS_0)$ be a boundary datum.
The functions taking values 
into the set~$\mcS_0$ will model the liquid crystal material. 

\smallskip We take $\mcP\subseteq\R^3$ to model a colloidal particle and aim to study the situation in which the colloids are small (in a suitable sense, to be clarified) and distributed throughout the ambient material, within $\Omega$. To this end we take the set of inclusions to be

\begin{equation} \label{inclusions}
 \mcP_\eps := \bigcup_{i=1}^{N_\eps} \mcP_\eps^i \qquad \textrm{and} 
 \qquad \mcP_\eps^i := x_\eps^i + \eps^\alpha R^i_\eps \mcP,
\end{equation}
where the~$x_\eps^i$'s are points in~$\Omega$, $\alpha>0$ is a fixed
($\eps$-independent) parameter, the~$R^i_\eps$'s are rotation matrices.

The liquid crystal material will then be located only outside the colloidal particles, that is within the set 
$$\Omega_\eps:=\Omega\setminus\mcP_\eps$$

The  material is described through functions $Q:\Omega_\eps\to \mcS_0$ that  minimise the following 
Landau-de Gennes free energy functional:
\begin{equation} \label{Eq:LdG}
 \mcF_\eps[Q] 
 := \int_{\Omega_\eps} \left(f_e(\nabla Q) + f_b(Q)\right)\d x +
 \eps^{3-2\alpha} \int_{\partial\mcP_\eps} f_s(Q, \, \nu) \, \d\sigma.
\end{equation} (where $\nu(x)$ denotes as usually the exterior normal at the point $x$ on the boundary).

In the above and the following we work in a non-dimensional setting (see for a general discussion  on non-dimensionalisation within liquid crystal context \cite{Gartland, NguyenZarnescu} and specifically for liquid crystal colloids \cite{ABL1}).

\bigskip The above energy functional has several terms. The most significant one, is the {\bf bulk potential} $f_b(Q)$ that models {\it the phase transition} from the liquid phase to the nematic phase. Physical invariances require the symmetry assumption $f_b(Q)=f_b(RQR^{\mathsf{T}})$ for any $Q\in\mcS_0$ and $R\in \mathrm{O}(3)$, and the most commonly used form is the Landau-de Gennes potential, up to fourth order, given as:
\begin{equation} \label{Eq:bulkLdG}
f_{LDG}(Q):= a \, \textrm{tr}(Q^2) - b \, \textrm{tr}(Q^3) 
+ c\left(\textrm{tr}(Q^2)\right)^2
\end{equation}
Here $a$, $b$, $c\in\R$ are material constants, with $a$ being proportional with the temperature. For $a$ large enough the global minimizer of such a potential is the zero matrix (which models the isotropic state) while for $a$ small enough the global minimizers are matrices of the form $\{s_+(n\otimes n-\frac{1}{3}\Id), n\in \mathbb{S}^2\}$ with $s_+$ an explicitly computable constant, depending on $a$, $b$, $c$ (see for instance \cite{Mottram} for more details). 
%
%

\smallskip The term $f_e(\nabla Q)$ models {\it the spatial variations} of the material. Physical invariances require the symmetry $f_e(D)=f_e(D^*)$ where we denote the third order tensor $D_{ijk}:=\frac{\partial Q_{ij}}{\partial x_k}$ and we have $D^*_{ijk}=R_{il}R_{jm}R_{kp}D_{lmp}$. Some terms satisfying these invariances are (where we denote $Q_{ij,k}:=\frac{\partial Q_{ij}}{\partial x_k}$ and assume summation over repeated indices):
\[
f_e^1:=Q_{ij,k}Q_{ij,k}, \quad f_e^2:=Q_{ij,k}Q_{ik,j},
\quad f_e^3=Q_{ij,j}Q_{ik,k}
\]
The most commonly used one is the first one above, that provides a reasonably good approximation in many cases of interest.
 
\smallskip Finally, {\it the effects induced by the particles} are modeled through the surface energy term that encodes the effect produced by the interaction between the boundary of the colloidal particles and the ambient fluid. The physical invariances require the following
 
\begin{equation} \label{invariance--}
  f_s(UQU^{\mathsf{T}}, \, Uu) = f_s(Q, \, u) \qquad 
  \textrm{for any } (Q, \, u)\in\mcS_0\times\R^3, \ U\in\O(3).
\end{equation}

The most commonly used surface energy is the so-called Rapini-Papoular type energy (see for instance \cite{Ravnik}), of the form:
\be
f_s(Q,\nu)=W \, \tr\left(Q-s_+\left(\nu\otimes \nu-\frac{1}{3}\Id\right)\right)^2
\ee
with $W>0$ a coefficient measuring the strength of the anchoring and the overall term measuring the deviation from the homeotropic (perpendicular) anchoring on the boundary.

\bigskip Taking the above into account, we will show that in the limit of $\eps\to 0$ the homogenized material will behave as if described by a limiting, homogenized, energy functional of the form 

\begin{equation} \label{Eq:homLdG-}
 \mcF_0[Q] 
 := \int_{\Omega_\eps} \left(f_e(\nabla Q) + f_b(Q) + f_{hom}(Q, \, x)\right)\d x.
\end{equation}

\bigskip As previously mentioned, our aim is to understand to what extent $f_b(Q) + f_{hom}(Q, \, x)$ can be {\it designed} to be an improvement of the original $f_b(Q)$. The control that we will allow is in suitably modifying the surface energy term, that models the effect of the colloids. We will develop a rather general machinery, that in particular will allow to show that given an arbitrary bulk term  $f_{LDG}$ as above, characterized by the constants $a,b,c$ we will be able to obtain the homogenized $f_{LDG}(Q)+f_{hom}(Q, \, x)$ to be of the Landau-de Gennes type, with constants $a',b'$ and $c'$, for arbitrary parameters $(a,b,c)$, respectively $(a',b',c')$. The homogenised bulk potential will be obtained  by suitably adjusting our {\it ``design parameter''}, the surface energy term $f_s$, depending on $a,b,c,a',b',c'$.

\bigskip The paper is organised as follows: 
in Section~\ref{Sec:main} we will present the technical assumptions and the main results. The proofs of the main results are divided between Section~\ref{Sec:prop}, where we study the properties of the functionals $\mcF_\eps$ for fixed~$\eps$, and Section~\ref{Sec:conv} where we provide the convergence of minimisers results. The applications to the standard types of Landau-de Gennes potentials, mentioned above, are presented at the end, in Section~\ref{Sec:LDG}.

\section{Technical assumptions and the main results}
\label{Sec:main}

Throughout the paper, we will denote by~$C$ a generic constant,
whose value may change from line to line, depending only on the domain,
the boundary datum and the free energy functional~\eqref{Eq:LdG},
but not on~$\eps$.
We will also write $A\lesssim B$ as a short-hand for~$A\leq C B$.

The target space of our maps is the set of $Q$-tensors,
that is, the set of symmetric, trace-free, real
$(3\times 3)$-matrices, which we denote~$\mcS_0$. For~$Q\in\mcS_0$, we let
$\abs{Q} := (\tr(Q^2))^{1/2} = (\sum_{i,j} Q_{ij})^{1/2}$; 
this defines a norm on the linear space~$\mcS_0$,
and this norm is induced by a scalar product.

\paragraph*{Assumptions.}

Let~$\Omega\subseteq\R^3$ be a bounded, smooth domain.
Let us consider a closed set $\mcP_\eps\subseteq\Omega$
of the form~\eqref{inclusions}, where~$\mcP$, 
$\alpha$, $x^i_\eps$ and~$R^i_\eps$ satisfy the assumptions below.

\begin{enumerate}[label=(H\textsubscript{\arabic*}), ref=H\textsubscript{\arabic*}]
 \item \label{hp:alpha} \label{hp:first} $1 < \alpha < 3/2$.
 
 \item \label{hp:Omega_eps}
 There exists a constant $\lambda_\Omega>0$ such that
 \[
  \dist(x_\eps^i, \, \partial\Omega) + 
  \frac{1}{2} \inf_{j\neq i} |x_\eps^j - x_\eps^i| \geq \lambda_\Omega\eps
 \]
 for any~$\eps>0$ and any~$i=1, \, \ldots, \, N_\eps$.
 
 \item \label{hp:conv} As~$\eps\to 0$, the measures
 $\mu_\eps := \eps^3\sum_{i=1}^{N_\eps}\delta_{x_\eps^i}$ 
 converge weakly$^*$ (as measures in~$\R^3$) 
 to $\d x\mres\Omega$, where~$\d x\mres\Omega$ 
 denotes the Lebesgue measure restricted to~$\Omega$.
 
 \item \label{hp:R} There exists a continuous 
 map~$R_*\colon\overline{\Omega}\to\SO(3)$
 such that $R_\eps^i = R_*(x_\eps^i)$ for any~$\eps>0$ 
 and any~$i=1, \, \ldots, \, N_\eps$.
 
 \item \label{hp:P} $\mcP\subseteq\R^3$ is a compact, convex set 
 that contains the origin. Moreover, there exists a 
 bijective, Lipschitz
 map~$\phi\colon\overline{B}_1\subseteq\R^3\to\mcP$,
 with Lipschitz inverse, such that $\phi(0) = 0$.
\end{enumerate}

Thanks to~\eqref{hp:alpha} and~\eqref{hp:Omega_eps},
the $\mcP_\eps^i$'s are pairwise disjoint when $\eps$ is small enough.
Under the assumption~\eqref{hp:Omega_eps},
the balls $B(x_\eps^i, \, \lambda_\Omega\eps)$
are contained in~$\Omega$ and pairwise disjoint; therefore, the number of
the inclusions is $N_\eps\lesssim\eps^{-3}$. Moreover, the total volume of
the inclusions converges to zero as~$\eps\to 0$, because
$|\mcP_\eps|\lesssim N_\eps \eps^{3\alpha} \lesssim \eps^{3(\alpha - 1)}$ and
$\alpha >1$ by assumption~\eqref{hp:alpha}. Thus, we are in the diluted
regime, as in \cite{ferromagnetic}. The assumption~\eqref{hp:conv}
guarantees that the inclusions are uniformly distributed,
at least for small values of~$\eps$.
Both the conditions~\eqref{hp:Omega_eps} and~\eqref{hp:conv}
are satisfied if, for instance,
the points $x^i_\eps$ are periodically distributed,
that is, if we choose
\[
 \{x_\eps^i\colon i=1, \, \ldots, \, N_\eps\}
 = \left\{y\in\Omega\colon \dist(y, \, \partial\Omega)\geq\eps
 \textrm{ and } y_k/\eps\in\Z \textrm{ for } k=1, \, 2, \, 3\right\} \! .
\]
The condition~\eqref{hp:R} guarantees that the orientation
of the inclusions is varying continuously across~$\overline{\Omega}$.
Finally, \eqref{hp:P} is compatible with a large class of 
shapes for the inclusions; for instance, \eqref{hp:P} is satisfied if $\mcP$ is
a sphere, an ellipsoid, a cube (or more generally, a convex polyhedron)
or a cylinder, with barycentre at the origin. Under the assumption~\eqref{hp:P},
the domain~$\Omega_\eps:= \Omega\setminus\mcP_\eps$ has Lipschitz boundary.

We consider the functional~$\mcF_\eps$, 
defined by~\eqref{Eq:LdG}. We assume the following conditions
on the energy densities~$f_e$, $f_b$, $f_s$.
We say that a function $f\colon\mcS_0\otimes\R^3\to\R$ 
is strongly convex if there exists~$\theta>0$ such that 
the function $\mcS_0\otimes\R^3\ni D \mapsto f(D) - \theta|D|^2$
is convex. 

\begin{enumerate}[label=(H\textsubscript{\arabic*}), ref=H\textsubscript{\arabic*}, resume]
 \item \label{hp:fe}
 $f_e\colon\mcS_0\otimes\R^3\to[0, \, +\infty)$ 
 is differentiable and strongly convex. Moreover, 
 there exists a constant~$\lambda_e>0$ such that 
 \[
  \lambda_e^{-1}|D|^2 \leq f_e(D) \leq \lambda_e|D|^2, \qquad
  |(\nabla f_e)(D)| \leq \lambda_e\left(|D| + 1\right)
 \]
 for any~$D\in\mcS_0\otimes\R^3$.
 
 \item \label{hp:fb} $f_b\colon\mcS_0\to\R$ is continuous, bounded from below 
 and there exists a positive constant~$\lambda_b$ such that
 $|f_b(Q)|\leq \lambda_b(|Q|^6 + 1)$ for any~$Q\in\mcS_0$.
 
 \item \label{hp:fs} \label{hp:last} $f_s\colon\mcS_0\times
 \mathbb{S}^2\to\R$ is continuous and there exists a
 constant~$\lambda_s>0$ such that, for any $Q_1$, 
 $Q_2\in\mcS_0$ and~$\nu\in\mathbb{S}^2$, there holds
 \[
  |f_s(Q_1, \, \nu) - f_s(Q_2, \, \nu)|\leq
  \lambda_s(|Q_1|^3 + |Q_2|^3 + 1) |Q_1 - Q_2|.
 \]
\end{enumerate}

As a consequence of~\eqref{hp:fs}, 
the surface energy density~$f_s$ has quartic growth in~$Q$.
Therefore, under the assumptions~\eqref{hp:fe}--\eqref{hp:fs},
the functional $\mcF_\eps$ is well-defined and finite 
on $H^1(\Omega_\eps, \mcS_0)$, thanks to the Sobolev embeddings
$H^1(\Omega_\eps)\hookrightarrow L^6(\Omega_\eps)$,
$H^{1/2}(\partial\Omega_\eps)\hookrightarrow L^4(\partial\Omega_\eps)$.

\paragraph*{The homogenised functional.}

Let us define the function~$f_{hom}\colon\mcS_0\times\overline{\Omega}\to\R$ by
\begin{equation} \label{f_hom}
 f_{hom}(Q, \, x) := \int_{\partial\mcP} 
 f_s(Q, \, R_*(x)\nu_{\mcP}) \, \d\sigma
\end{equation}
for any~$(Q, \, x)\in\mcS_0\times\overline{\Omega}$,
where~$\nu_{\mcP}$ denotes the 
\emph{inward}-pointing unit normal to~$\partial\mcP$,
and~$R_*\colon\overline{\Omega}\to\SO(3)$ is the map
given by~\eqref{hp:R}. Under the assumptions~\eqref{hp:R} and~\eqref{hp:fs},
it is easily checked that $f_{hom}$ is continuous 
on~$\mcS_0\times\overline{\Omega}$ and has quartic growth in~$Q$.
Our candidate homogenised functional is defined
for any $Q\in H^1(\Omega, \mcS_0)$ as
\begin{equation} \label{Eq:homLdG}
 \mcF_0[Q] 
 := \int_{\Omega_\eps} \left(f_e(\nabla Q) 
 + f_b(Q) + f_{hom}(Q, \, x)\right)\d x.
\end{equation}

\paragraph*{Main results.}

The main result of these notes concerns the asymptotic behaviour,
as~$\eps\to 0$, of \emph{local} minimisers of the functional~$\mcF_\eps$.

Let~$g\in H^{1/2}(\partial\Omega, \, \mcS_0)$ be a boundary datum.
By a slight abuse of notation, we denote by~$H^1_g(\Omega_\eps, \mcS_0)$ ---
respectively, $H^1_g(\Omega, \, \mcS_0)$ ---
the set of maps~$Q\in H^1(\Omega_\eps, \mcS_0)$
--- respectively, $Q\in H^1(\Omega, \mcS_0)$ --- that satisfy~$Q = g$
on~$\partial\Omega$, in the sense of traces.
(We do not prescribe a boundary value for~$Q$ 
on~$\partial\Omega_\eps\setminus\partial\Omega = \partial\mcP_\eps$.)
We let $E_\eps\colon H^1_g(\Omega_\eps, \mcS_0)\to H^1_g(\Omega, \mcS_0)$ 
denote the harmonic extension operator, defined in the following
way: for~$Q\in H^1(\Omega_\eps, \mcS_0)$,
$E_\eps Q := Q$ on~$\Omega_\eps$ and, for each~$i$,
$E_\eps Q$ on~$\mcP_\eps^i$ is the unique solution of the problem
\[
 \begin{cases}
  -\Delta (E_\eps Q) = 0 & \textrm{in } \mcP_\eps^i \\
  E_\eps Q = Q           & \textrm{on } \partial\mcP_\eps^i.
 \end{cases}
\]

\begin{theorem}  \label{th:loc-min}
 Suppose that the assumptions \eqref{hp:first}--\eqref{hp:last}
 are satisfied. Suppose, moreover, that $Q_0\in H^1_g(\Omega, \mcS_0)$
 is an isolated $H^1$-local minimiser for 
 $\mcF_0$ --- that is, there exists $\delta_0>0$ such that
 \begin{equation*} 
  \mcF_0[Q_0] < \mcF_0[Q]
 \end{equation*}
 for any~$Q\in H^1_g(\Omega, \mcS_0)$ such that $Q\neq Q_0$ and
 $\|Q - Q_0\|_{H^1(\Omega)}\leq\delta_0$. Then, for any~$\varepsilon$
 small enough, there exists an $H^1$-local minimiser $Q_\eps$ for~$\mcF_\eps$
 such that $E_\eps Q_\eps\to Q_0$ strongly in~$H^1(\Omega)$ as~$\eps\to 0$.
\end{theorem}

\begin{remark}
 Theorem~\ref{th:loc-min} applies, in particular, to any critical point~$Q_0$
 of~$\mcF_0$ that is locally (strictly) stable, that is, satisfies
 \[
  \frac{\d^2}{\d t^2}_{|t=0} \mcF_0[Q_0 + t P]>0
 \]
 for any~$P\in H^1_0(\Omega, \mcS_0)$ with~$P\not\equiv 0$.
\end{remark}

\begin{remark} \label{remark:trade-off}
 There is a trade-off between the growth of the
 surface energy density, $f_s$, and the parameter~$\alpha$.
 If~$f_s$ is allowed to be
 a function of quartic growth in~$Q$, as in Assumption~\eqref{hp:fs}, then
 Theorem~\ref{th:loc-min} cannot hold when $\alpha > 3/2$
 (see Lemma~\ref{lem:unbounded3} below).
 Other regimes for the parameter $\alpha$ can be considered, if
 we impose different growth conditions on~$f_s$; 
 an example where~$f_s$ is a quadratic function and~$1 < \alpha < 3$
 was studied in~\cite{ferromagnetic}.
\end{remark}

We cannot study the asymptotic behaviour of global minimisers of~$\mcF_\eps$
because, under the assumptions~\eqref{hp:first}--\eqref{hp:last},
the functional~$\mcF_\eps$ might
be unbounded from below (see Lemma~\ref{lem:unbounded} below). 
However, we can provide a statement about global
minimisers of~$\mcF_\eps$ under stronger assumptions.

\begin{theorem}  \label{th:global-min}
 Suppose that the assumptions \eqref{hp:first}--\eqref{hp:last}
 are satisfied. In addition, suppose that there exist positive 
 constants~$\mu$ and $C$ such that
 \[
  f_b(Q) \geq \mu |Q|^6 - C \qquad 
  \textrm{for any } Q\in\mcS_0.
 \]
 Then, for~$\varepsilon$ small enough, there exists a global 
 minimiser $Q_\eps$ for~$\mcF_\eps$ in $H^1_g(\Omega_\eps, \mcS_0)$.
 Moreover, up to a (non-relabelled) subsequence, $E_\eps Q_\eps$ converges 
 strongly in~$H^1(\Omega)$ to a global minimiser 
 for~$\mcF_0$ in~$H^1_g(\Omega, \mcS_0)$.
\end{theorem}

\begin{remark} The bulk potential should satisfy the symmetry requirement $f_b(Q)=f_b(RQR^{\mathsf{T}})$ for any $Q\in\mcS_0$ and $R\in O(3)$ and as such it can be shown that it is a function of $\tr(Q^2)$ and $\tr(Q^3)$. The most commonly used potential, the one we mentioned before, the quartic Landau-de Gennes
potential~\eqref{Eq:bulkLdG}  is the simplest Taylor expansion type of bulk potential that does satisfy the symmetry and it is cut at fourth order, because this is the lowest order term that predicts as a global minimizer a uniaxial phase (i.e. the global minimizer has two equal eigenvalues).
However, it does \emph{not} satisfy the assumptions 
of Theorem~\ref{th:global-min}. Nevertheless, Theorem~\ref{th:global-min}
\emph{does} apply to the sextic Landau-de Gennes potential that can be relevant for the so-called biaxial minimizers (i.e. all eigenvalues are distinct), see~\cite{deGennes}, Sec. $2.3.3$:
\[
 \begin{split}
  f_b(Q) &= a_2\,\tr(Q^2) - a_3\,\tr(Q^3) + a_4\left(\tr(Q^2)\right)^2 \\
  &\qquad\qquad + a_5\,\tr(Q^2)\tr(Q^3) + a_6\left(\tr(Q^2)\right)^3 + a_6^\prime\left(\tr(Q^3)\right)^2,
 \end{split}
\]
so long as $a_6>0$ and $6a_6 + a_6^\prime>0$. 
\end{remark}

\paragraph*{Additional notation.}

We define
\begin{equation} \label{eq:J_eps}
 J_\eps[Q] := \eps^{3-2\alpha} \int_{\partial\mcP_\eps} 
 f_s(Q, \, \nu)\, \d\sigma
 \qquad \textrm{for } Q\in H^1_g(\Omega_\eps, \mcS_0)
\end{equation}
and
\begin{equation} \label{eq:J_0}
 J_0[Q] := \int_\Omega f_{hom} (Q, \, x) \, \d x \qquad
 \textrm{for } Q\in H^1_g(\Omega, \mcS_0).
\end{equation}
The functional~$J_0$ is the candidate limit of the surface integrals, $J_\eps$.

\section{Properties of the functional~$\mcF_\eps$}
\label{Sec:prop}
\subsection{Analytical tools: Trace and extension}

One of the main tools we will use in the sequel is the following
$L^p$-trace inequality, which is adapted 
from~\cite[Lemma~4.1]{ferromagnetic}.
Given a set~$\mcP\subseteq\R^3$ that contains the origin and a number~$a>0$, 
we set $a\mcP :=\{ax\colon x\in\mcP\}$.

\begin{lemma} \label{lem:trace}
 Let~$\mcP\subseteq\R^3$ be a closed, convex set that satisfies~\eqref{hp:P}.
 Let~$p\in [2, \, 4]$. Then, there exists $C = C(\mcP, \, \phi, \, p)>0$
 such that, for any~$a>0$, $b\geq 2a$ and any~$u\in H^1(b\mcP\setminus a\mcP)$, there holds
 \[
  \int_{\partial(a\mcP)} \abs{u}^p \d\sigma \leq C 
  \int_{b\mcP\setminus a\mcP} \left(\abs{\nabla u}^2 + \abs{u}^{2p-2}\right)\d x
  + \frac{Ca^2}{b^3} \int_{b\mcP\setminus a\mcP} \abs{u}^p\d x.
 \]
\end{lemma}
\begin{proof}
 Let~$\phi\colon\overline{B}_1\subseteq\R^3\to\mcP$
 be a bijective, Lipschitz map, with Lipschitz inverse,
 such that~$\phi(0) = 0$; such a map exists 
 by the assumption~\eqref{hp:P}.
 The map~$\phi_*(x) := |x|\phi(x/|x|)$,
 defined for $x\in\R^3\setminus\{0\}$,
 is Lipschitz 
 and maps $\bar{B}_b\setminus\bar{B}_a$ 
 onto~$b\mcP\setminus a\mcP$. 
 If~$t_1$, $t_2$ are positive numbers and 
 $y_1\in\partial\mcP$, $y_2\in\partial\mcP$ are such 
 that $t_1y_1 = t_2 y_2$, then $y_1 = y_2$ because~$\mcP$
 is convex and the origin lies in the interior of~$\mcP$;
 as a consequence, $\phi_*$ is injective. Finally,
 $\|\nabla\phi_*\|_{L^\infty}$
 and $\|\nabla(\phi_*^{-1})\|_{L^\infty}$ are bounded by a constant
 that depends on~$\phi$ and~$\mcP$, but not on~$a$, $b$.
 Therefore, up to composition with~$\phi_*^{-1}$,
 we can assume without loss of generality that $\mcP = \bar{B}_1$.
 
 Having reduced to the case~$\mcP$ is a ball, we can use spherical coordinates:
 \[
  x_1 = \rho\cos\theta\sin\varphi, \quad x_2 = \rho\sin\theta\sin\varphi,
  \quad x_3 = \rho\cos\varphi
 \]
 where~$\rho\in [a, \, b]$, $\theta\in [0, \, 2\pi]$, $\varphi\in [0, \, \pi]$.
 For $u\in H^1(b\mcP\setminus a\mcP)$ and a.e. $(\rho, \, \theta, \, \varphi)$,
 there holds
 \[
  \begin{split}
   \abs{u}^p(a, \, \theta, \, \varphi) &= \abs{u}^p(\rho, \, \theta, \, \varphi) 
      - \int_a^\rho \partial_\rho\left(\abs{u}^p\right) (s, \, \theta, \, \varphi)\, \d s \\
   &\leq \abs{u}^p(\rho, \, \theta, \, \varphi) 
      + p \int_a^\rho \abs{u}^{p-1}\abs{\partial_\rho u} (s, \, \theta, \, \varphi)\, \d s \\
   &\leq \abs{u}^p(\rho, \, \theta, \, \varphi) + \frac{p}{2} \int_a^\rho 
      \left(\abs{\partial_\rho u}^2 + \abs{u}^{2p-2} \right) \d s \\
   &\leq \abs{u}^p(\rho, \, \theta, \, \varphi) 
     + \frac{p}{2a^2} \int_a^\rho \left( \abs{\partial_\rho u}^2 + \abs{u}^{2p-2} \right) s^2 \d s.
  \end{split}
 \]
 We divide both sides of this inequality by $\rho^2\sin\varphi$
 (i.e. the Jacobian of the change of coordinates)
 and integrate with respect to 
 $(\rho, \, \theta, \, \varphi)\in (a, \, b)\times(0, \, 2\pi)\times(0, \, \pi)$.
 We obtain
 \[
  \begin{split}
   &\frac{b^3 - a^3}{3} \int_0^{2\pi}\int_0^\pi \abs{u}^p(a, \, \theta, \, \varphi)
     \sin\varphi \, \d\theta\,\d\varphi
     \leq \int_{B_b\setminus B_a}\abs{u}^p\d x  \\
   &\qquad\qquad\qquad + \frac{p(b^3 - a^3)}{6a^2} \int_{B_b\setminus B_a} 
     \left(\abs{\partial_\rho u}^2 + \abs{u}^{2p-2} \right) \d x.
  \end{split}
 \]
 Since the surface element on the sphere~$\partial B_a$ is $\d\sigma = a^2\sin\varphi \, \d\theta\,\d\varphi$,
 the previous inequality may be rewritten as
 \[
  \frac{b^3 - a^3}{3a^2} \int_{\partial B_a} \abs{u}^p \d\sigma
     \leq \int_{B_b\setminus B_a}\abs{u}^p\d x + \frac{p(b^3 - a^3)}{6a^2} \int_{B_b\setminus B_a} 
     \left(\abs{\partial_\rho u}^2 + \abs{u}^{2p-2} \right) \d x.
 \]
 We multiply both sides of this inequality by $(b^3-a^3)/(3a^2)$. By taking into account 
 that $b^3 - a^3\gtrsim b^3$, because we have assumed that $b\geq 2a$, the lemma follows. 
\end{proof}

\begin{lemma} \label{lem:trace_eps}
 For any~$Q\in H^1(\Omega_\eps, \mcS_0)$ and any~$p\in[2, \, 4]$, there holds
 \[
  \eps^{3-2\alpha} \int_{\partial\mcP_\eps} \abs{Q}^p \d\sigma \lesssim 
    \eps^{3 - 2\alpha} \int_{\Omega_\eps} \left(\abs{\nabla Q}^2 + \abs{Q}^{2p-2}\right)\d x
    + \int_{\Omega_\eps} \abs{Q}^p\d x.
 \]
\end{lemma}
\begin{proof}
 For each~$i$, consider the inclusion~$\mcP_\eps^i = x_\eps^i + \eps^\alpha R_\eps^i\mcP$
 and~$\widehat{\mcP}_\eps^i := x_\eps^i + \mu\eps R_\eps^i\mcP$, where $\mu>0$
 does not depend on~$i$, $\eps$. For~$\eps$ small enough we have
 $\mcP_\eps^i\csubset\widehat{\mcP}_\eps^i$ and, thanks to assumption~\eqref{hp:Omega_eps},
 by taking $\mu$ small enough we can make sure that 
 the $\widehat{\mcP}_\eps^i$'s are pairwise disjoint. Therefore,
 we can apply Lemma~\ref{lem:trace} on each $\widehat{\mcP}_\eps^i\setminus\mcP_\eps^i$
 (with~$a = \eps^\alpha$, $b=\mu\eps$) and sum the corresponding inequalities.
 This proves the lemma.
\end{proof}

It will be useful to consider maps defined on the fixed domain~$\Omega$
instead of~$\Omega_\eps$. To this end, we consider the harmonic extension operator
$E_\eps\colon H^1_g(\Omega_\eps, \mcS_0)\to H^1_g(\Omega, \mcS_0)$.
Given~$Q\in H^1_g(\Omega_\eps, \mcS_0)$ we let $E_\eps Q := Q$ on~$\Omega_\eps$
and, inside each inclusion~$\mcP_\eps^i$, we define $E_\eps Q$ as the unique solution of
\[
 \begin{cases}
  -\Delta (E_\eps Q) = 0 & \textrm{in } \mcP_\eps^i\\
  E_\eps Q = Q         & \textrm{on } \partial\mcP_\eps^i.
 \end{cases}
\]
The operator~$E_\eps$ is linear and uniformly bounded with respect to~$\eps>0$, 
as demonstrated by the following

\begin{lemma} \label{lem:extension}
 There exists a constant $C>0$ such that 
 $\|\nabla (E_\eps Q)\|_{L^2(\Omega)} \leq C \|\nabla Q\|_{L^2(\Omega_\eps)}$
 for any~$\eps>0$ and any $Q\in H^1_g(\Omega_\eps, \mcS_0)$.
\end{lemma}
\begin{proof} 
 Consider a single inclusion $\mcP_\eps^i = x_\eps^i + \eps^\alpha R_\eps^i\mcP$ 
 and let~$\widetilde{\mcP}_\eps^i := x_\eps^i + 2\eps^\alpha R_\eps^i\mcP$.
 Thanks to~\eqref{hp:Omega_eps}, we know that the $\tilde{\mcP}_\eps^i$'s 
 are pairwise disjoint for~$\eps$ small enough. Therefore, 
 it suffices to prove that
 \begin{equation} \label{harm_extension}
  \|\nabla (E_\eps Q)\|^2_{L^2(\mcP_\eps^i)} \leq 
  C \|\nabla Q\|^2_{L^2(\widetilde{\mcP}_\eps^i\setminus\mcP_\eps^i)}.
 \end{equation}
 This inequality is scale-invariant and hence, by a scaling argument,
 we can assume without loss of generality that $\eps = 1$, $\mcP_\eps^i = \mcP$,
 $\widetilde{\mcP}_\eps^i = 2\mcP$. Then, elementary properties of Lapace's equation 
 and the trace inequality yield
 \[
  \|\nabla (E_\eps Q)\|^2_{L^2(\mcP)} \lesssim [Q]^2_{H^{1/2}(\partial\mcP)}
  \lesssim \|\nabla Q\|^2_{L^2(2\mcP\setminus\mcP)}. \qedhere
 \]
\end{proof}

\subsection{Equicoercivity of the~$\mcF_\eps$'s}

As we will see below, the 
assumptions~\eqref{hp:first}--\eqref{hp:last} are \emph{not} enough
to guarantee coercivity of~$\mcF_\eps$. However, it is possible to restore coercivity
under an additional assumption on~$f_b$, namely,
if there exist positive constants~$\mu$, $C$ such that
\begin{equation} \label{hp:coercivity}
 f_b(Q) \geq \mu |Q|^6 - C \qquad \textrm{for any } Q\in\mcS_0.
\end{equation}

\begin{prop} \label{prop:coercivity}
 Suppose that the assumptions~\eqref{hp:first}--\eqref{hp:last}
 and~\eqref{hp:coercivity} are satisfied. Let~$Q\in H^1_g(\Omega_\eps, \mcS_0)$
 satisfy~$\mcF_\eps[Q]\leq M$, for some ($\eps$-independent) constant~$M$. Then, there holds
 \[
  \int_{\Omega_\eps} \abs{\nabla Q}^2 \d x \leq M^\prime
 \]
 for~$\varepsilon>0$ small enough and for some~$M^\prime>0$
 depending only on~$M$, $f_e$, $f_b$, $f_s$, $\Omega$, $\mcP$.
\end{prop}
\begin{proof}
 Recall that, due to the Assumption~\eqref{hp:fs}, there holds~$|f_s(Q)|\lesssim |Q|^4 + 1$.
 By applying Lemma~\ref{lem:trace_eps} (with the choice~$p=4$) and 
 the H\"older inequality, we obtain that
 \[
  \begin{split}
   J_\eps[Q] 
   &\geq - C_1\eps^{3 - 2\alpha} \int_{\Omega_\eps} \left(\abs{\nabla Q}^2 + \abs{Q}^{6}\right)\d x
    - C_1\int_{\Omega_\eps} \abs{Q}^4\d x - C \\
   &\geq - C_1\eps^{3 - 2\alpha} \int_{\Omega_\eps} \left(\abs{\nabla Q}^2 + \abs{Q}^{6} \right)\d x
    - C_2 \left(\int_{\Omega_\eps} \abs{Q}^{6}\d x\right)^{2/3} - C 
  \end{split}
 \]
 for some positive constants~$C_1$, $C_2$, $C$.
 On the other hand, the assumptions \eqref{hp:fe} and~\eqref{hp:coercivity} give
 \[
  \int_{\Omega_\eps} \left(f_e(\nabla Q) + f_b(Q)\right) \geq 
  C_3 \int_{\Omega_\eps} \left(\abs{\nabla Q}^2 + |Q|^{6}\right)\, \d x - C.
 \]
 Therefore, the energy bound $\mcF_\eps[Q]\leq M$ implies
 \[
  (C_3 - C_1\eps^{3-2\alpha}) \int_{\Omega_\eps} \left(\abs{\nabla Q}^2 + |Q|^{6}\right)\, \d x
  \leq C_2 \left(\int_{\Omega_\eps} \abs{Q}^{6}\d x\right)^{2/3} + C.
 \]
 Since~$\alpha<3/2$, the coefficient $C_3 - C_1\eps^{3-2\alpha}$ 
 is strictly positive for~$\eps$ small enough. The proposition follows by 
 observing that the left-hand side has linear growth in
 $\int_{\Omega_\eps}(|\nabla Q|^2 + |Q|^6)\, \d x$
 while the right-hand side has sublinear growth,
 hence the integral must be bounded.
\end{proof}

In the rest of this section, we illustrate some patologies
that may lead to the loss of equicoercivity.
If the assumption~\eqref{hp:coercivity} is not satisfied, the functional~$\mcF_\eps$
might be unbounded from below. To illustrate this phenomenon,
we consider an example that involves scalar functions. 
For simplicity, we assume that the inclusions are balls, i.e. $\mcP = B_1$ and
$\mcP_\eps = \cup_{i=1}^{N_\eps} B(x_\eps^i, \, \eps^\alpha)$.
For~$u\in H^1(\Omega_\eps, \, \R)$, define the functional
\begin{equation} \label{example}
 {\mycal G}_\eps(u) := \int_{\Omega_\eps} \left(|\nabla u|^2 + k |u|^q\right)\d x
 - \delta\eps^{3-2\alpha} \int_{\partial\mcP_\eps} |u|^p\,\d\sigma,
\end{equation}
where~$k$, $\delta$, $p$, $q$ are positive parameters.

\begin{lemma} \label{lem:unbounded}
 For any~$k>0$, $\delta>0$, $p>2$, $q<2p-2$, $\alpha>1$ and for $\varepsilon>0$ small enough,
 we have $\inf\{\mycal G_\eps(u)\colon u\in H^1(\Omega_\eps), \ u=0 \textrm{ on } \partial\Omega\} = -\infty$.
\end{lemma}
\begin{proof}
 Let~$U := B_2\setminus B_1$ be a spherical shell of inner radius~$1$ and outer radius~$2$.
 Consider the function
 \[
  v(x) := M(2 - |x|)^\gamma \qquad \textrm{for } x\in U,
 \]
 where $M\geq 1$, $\gamma\geq 1$ are parameters to be specified later on. We can explicitely compute
 the energy of~$v$:
 \begin{gather*}
  \int_U\abs{\nabla v}^2\d x = 4\pi\gamma^2M^2  \int_1^2(2-\rho)^{2\gamma-2}\rho^2\, \d\rho
  \leq \frac{16\pi\gamma^2 M^2}{2\gamma-1} \\
  \int_U \abs{v}^q \d x = 4\pi M^q  \int_1^2(2-\rho)^{q\gamma}\rho^2\, \d\rho
  \leq \frac{16\pi M^q}{q\gamma + 1}, \quad
  \int_{\partial B_1} \abs{v}^p \d\sigma = 4\pi M^p.
 \end{gather*}
 Now, choose one of the inclusions, say~$B(x_\eps^1, \eps^\alpha)$. Thanks to \eqref{hp:Omega_eps},
 for~$\eps$ small enough we have $B(x_\eps^1, 2\eps^\alpha)\setminus 
 \bar{B}(x_\eps^1, \eps^\alpha)\subseteq\Omega_\eps$. We define
 $u_\eps(x) := v(\eps^{-\alpha}(x-x_\eps^1))$ if~$x\in B(x_\eps^1, 2\eps^\alpha)\setminus 
 \bar{B}(x_\eps^1, \eps^\alpha)$ and~$u_\eps(x) := 0$ otherwise. By scaling, we have
 \begin{equation} \label{unbounded1}
  \begin{split}
   {\mycal G}_\eps(u_\eps) &= \eps^\alpha \int_U\abs{\nabla v}^2\d x + k \eps^{3\alpha} \int_U \abs{v}^q \d x
   - \delta \eps^{3} \int_{\partial B_1} \abs{v}^p\d\sigma \\
   & \leq C_{\eps,1} \left(\frac{\gamma^2 M^2}{2\gamma - 1} + \frac{M^q}{q\gamma + 1} \right) - C_{\eps,2} M^p,
  \end{split}
 \end{equation}
 where $C_{\eps,1}$, $C_{\eps,2}$ are positive numbers depending on $\eps$, $\alpha$, $k$, $\delta$, $q$.
 (Actually we have $C_{\eps,2} \ll C_{\eps,1}$ as~$\eps\to 0$,
 but this is irrelevant because, for the purposes of this lemma, $\eps$ is fixed.)
 We take~$\gamma = M^\beta$, with~$\beta> 0$. If we can choose~$\beta> 0$ in such a way that
 \begin{equation} \label{unbounded2}
  \max\{2 + \beta, \, q - \beta\} < p
 \end{equation}
 then, by letting~$M\to+\infty$ in~\eqref{unbounded1}, we see that ${\mycal G}_\eps$
 is unbounded from below. Now, \eqref{unbounded2} is equivalent to
 $q - p <\beta<p - 2$ and, because we have assumed that $p>2$, $q<2p-2$, we
 can find~$\beta>0$ satisfying~\eqref{unbounded2}.
\end{proof}

If Condition~\eqref{hp:coercivity} is met, but $\alpha > 3/2$ 
so that the exponent in~$\eps^{3-2\alpha}$ becomes negative, the energy 
bounds from below may degenerate as~$\eps\to 0$ and equicoercivity may be lost.
As an example, we consider the same functional as in~\eqref{example} 
with~$q = 2p-2$, namely:
\[
 {\mycal G}_\eps(u) := \int_{\Omega_\eps} \left(|\nabla u|^2 + k |u|^{2p-2}\right)\d x
 - \delta\eps^{3-2\alpha} \int_{\partial\mcP_\eps} |u|^p\,\d\sigma
\]
for~$u\in H^1(\Omega_\eps, \, \R)$.

\begin{lemma} \label{lem:unbounded3}
 For any $2 < p < 4$, $3/2 < \alpha < 6/p$, $k>0$, $\delta>0$, there holds
 $\inf\{{\mycal G}_\eps(u)\colon u\in H^1(\Omega_\eps), 
  \ u=0 \textrm{ on } \partial\Omega\} \to -\infty$ as~$\eps\to 0$.
\end{lemma}
\begin{proof}
 Let~$\varphi\in C^\infty_{\mathrm{c}}(B_2)$ be a non-negative function such that 
 $\varphi = 1$ on~$B_1$, and let~$\beta > 0$ be a parameter to be chosen later.
 Choose one of the inclusions, say~$B(x_\eps^1, \eps^\alpha)$,
 and define $u_\eps\in H^1(\Omega)$ by
 \[
  u_\eps(x) := \eps^{-\alpha/2 - \beta} \varphi(\eps^{-\alpha}( x - x_\eps^1)) 
  \qquad \textrm{if } x\in B(x_\eps^1, 2\eps^\alpha)\setminus B(x_\eps^1, \eps^\alpha)
 \]
 and~$u_\eps(x) := 0$ otherwise. By means of a change of variables, we compute that 
 \[
  {\mycal G}_\eps(u_\eps)
  = \eps^{-2\beta} \int_{B_2\setminus B_1} \abs{\nabla\varphi}^2\d x 
  + k\eps^{-\beta(2p - 2) + \alpha(4-p)} \int_{B_2\setminus B_1} \abs{\varphi}^{2p-2}\d x
  - 4\pi\delta\eps^{3 - \alpha p/2 - \beta p}
 \]
 Now, if we are able to choose~$\beta>0$ in such a way that
 \[
   \alpha p/2 + \beta p - 3 > \max\left\{2\beta, \, \beta(2p - 2) - \alpha(4-p) \right\} \!,
 \]
 then the lemma follows. This condition can be equivalently rewritten as
 \begin{equation} \label{beta-cond}
  \frac{6 - \alpha p}{2p - 4} < \beta < \frac{-\alpha p + 8\alpha - 6}{2p - 4}
 \end{equation}
 It can be checked the system~\eqref{beta-cond} admits a
 positive solution, because~$p>2$ and~$3/2 < \alpha < 6/p$.
\end{proof}

For larger values of the parameter~$\alpha$,
even on bounded subsets of~$H^1(\Omega)$ the minimal energy 
may be unbounded in the limit as~$\eps\to 0$. Again, 
we suppose that $\mcP_\eps = \cup_{i=1}^{N_\eps} B(x_\eps^i, \, \eps^\alpha)$,
and we consider the functional
\[
 {\mycal G}_\eps(u) := \int_{\Omega_\eps} \left(|\nabla u|^2 + k |u|^6\right)\d x
 - \delta\eps^{3-2\alpha} \int_{\partial\mcP_\eps} |u|^p\,\d\sigma
\]
for~$u\in H^1(\Omega_\eps, \, \R)$.

\begin{lemma} \label{lem:unbounded2}
 Suppose that $2 < p \leq 4$ and $\alpha > 6/p$. Then, for any $M>0$, there holds
 \[
  \inf\left\{{\mycal G}_\eps(u)\colon u\in H^1(\Omega_\eps), 
  \ u=0 \textrm{ on } \partial\Omega, \ \|E_\eps u\|_{H^1(\Omega)}\leq M
  \right\} \to -\infty \qquad \textrm{as } \eps\to 0.
 \]
\end{lemma}

Lemma~\ref{lem:unbounded2} shows, in particular,
that the assumption~$\alpha < 3/2$ is sharp if the surface energy term
has quartic growth.

\begin{proof}
 We consider a variant of the example given in Lemma~\ref{lem:unbounded3}.
 Let~$u_\eps$ be defined by
 \[
  u_\eps(x) := \eta \eps^{-\alpha/2} \varphi(\eps^{-\alpha}( x - x_\eps^i)) 
  \qquad \textrm{if } x\in B(x_\eps^1, 2\eps^\alpha)\setminus B(x_\eps^1, \eps^\alpha)
 \]
 and~$u_\eps(x) := 0$ otherwise. Here,
 $\varphi$ is as in Lemma~\ref{lem:unbounded3} and $\eta$ is a positive parameter.
 By a change of variables, we see that $\|\nabla u_\eps\|_{L^2(\Omega_\eps)} 
 = \eta\|\nabla\varphi\|_{L^2(B_2\setminus B_1)} < + \infty$
 and hence, by taking~$\eta$ small enough and applying Lemma~\ref{lem:extension},
 we can make sure that $\|E_\eps u_\eps\|_{H^1(\Omega)} \leq M$.
 On the other hand, we have
 \[
  {\mycal G}_\eps(u_\eps)
  = \int_{B_2\setminus B_1} \left(\eta^2\abs{\nabla\varphi}^2 + k\eta^6\varphi^6\right)\d x
   - 4\pi\delta \eta^p \eps^{3 - \alpha p/2} 
 \]
 and the right-hand side tends to~$-\infty$ as~$\eps\to 0$, because $\alpha < 6/p$.
\end{proof}

\subsection{Lower semi-continuity of the~$\mcF_\eps$'s}

As we will see in Lemma~\ref{lem:not_lsc} below,
the assumptions~\eqref{hp:first}--\eqref{hp:last}
do \emph{not} guarantee that $\mcF_\eps$ is 
sequentially lower semi-continuous with respect to
the weak topology in~$H^1(\Omega_\eps)$. However, we can 
prove weak sequential lower semi-continuity on 
\emph{bounded} subsets of~$H^1(\Omega_\eps)$, for small~$\eps$.

\begin{prop} \label{prop:lsc}
 Suppose that the assumptions~\eqref{hp:first}--\eqref{hp:last} are satisfied.
 For any positive~$M$ there exists~$\eps_0(M)>0$ with the following property:
 if~$0 <\varepsilon\leq\varepsilon_0(M)$ and
 if~$(Q_j)_{j\in\N}$ is a sequence in~$H^1(\Omega_\eps, \mcS_0)$ 
 that converges $H^1$-weakly to~$Q\in H^1(\Omega_\eps, \mcS_0)$
 and satisfies
 \begin{equation} \label{hp:bound}
  \|\nabla Q_j\|_{L^2(\Omega_\eps)} \leq M,
 \end{equation}
 then
 \[
  \mcF_\eps[Q] \leq \liminf_{j\to+\infty} \mcF_\eps[Q_j].
 \]
\end{prop}
\begin{proof}
 We analyse separately the terms in~$\mcF_\eps$, 
 starting by the gradient term~$f_e$. Let us define
 \begin{equation*}
  \mu := \liminf_{j\to+\infty} \int_{\Omega_\eps} \abs{\nabla Q_j}^2 \, \d x - 
     \int_{\Omega_\eps} \abs{\nabla Q}^2 \, \d x.
 \end{equation*}
 By the $H^1$-weak convergence $Q_j\rightharpoonup Q$, we know that~$\mu\geq 0$.
 Moreover, up to extracting a subsequence, we can assume that 
 \begin{equation} \label{lsc0}
  \int_{\Omega_\eps} \abs{\nabla Q_j}^2 \, \d x \to 
  \int_{\Omega_\eps} \abs{\nabla Q}^2 \, \d x + \mu
  \qquad \textrm{as } j\to+\infty.
 \end{equation}
 Since~$f_e$ is assumed to be strongly
 convex~\eqref{hp:fe}, for~$\theta>0$ small enough the 
 function~$\tilde{f}_e$ defined by $\tilde{f}_e(D) := f_e(D) - \theta|D|^2$
 for~$D\in\mcS_0\otimes\R^3$ is convex. Therefore, using~\eqref{lsc0}, we obtain
 \begin{equation} \label{lsc1}
  \begin{split}
   &\liminf_{j\to+\infty} \int_{\Omega_\eps} f_e(\nabla Q_j) \, \d x - 
     \int_{\Omega_\eps} f_e(\nabla Q) \, \d x \\ 
   & \qquad\qquad = \liminf_{j\to+\infty} 
     \int_{\Omega_\eps} \tilde{f}_e(\nabla Q_j) \, \d x - 
     \int_{\Omega_\eps} \tilde{f}_e(\nabla Q) \, \d x + \theta \mu
   \geq \theta\mu.
  \end{split}
 \end{equation}

 We consider now the bulk term~$f_b$. Due to the compact Sobolev embedding
 $H^1(\Omega_\eps)\hookrightarrow L^2(\Omega_\eps)$, by extracting a subsequence 
 we can assume that $Q_j\to Q$ a.e. in~$\Omega_\eps$. Then, since~$f_b$
 is assumed to be continuous and bounded from below \eqref{hp:fb},
 we can apply Fatou's lemma to obtain
 \begin{equation} \label{lsc2}
  \begin{split}
   &\liminf_{j\to+\infty} \int_{\Omega_\eps} f_b(Q_j) \, \d x - 
     \int_{\Omega_\eps} f_b(Q) \, \d x \geq 0. 
  \end{split}
 \end{equation}
 
 Finally, we deal with the surface integral, $J_\eps$.
 For any~$j\in\N$, we define two sets:
 \begin{gather*}
  A_j := \left\{x\in\partial\mcP_\eps\colon
   |Q_j(x) - Q(x)| \leq |Q(x)| + 1\right\} \!, \\
  B_j := \partial\mcP_\eps\setminus A_j = 
   \left\{x\in\partial\mcP_\eps\colon |Q_j(x) - Q(x)| > |Q(x)| + 1\right\} \!.
 \end{gather*}
 We first consider the set~$A_j$, where the inequality
 $\abs{Q_j-Q}\leq\abs{Q}+1$ holds.
 We apply the assumption~\eqref{hp:fs} and deduce that,
 a.e. on~$A_j$, there holds
 \[
  \begin{split}
   \abs{f_s(Q_j, \, \nu) - f_s(Q, \, \nu)} &\lesssim \left(|Q_j|^3 + |Q|^3 + 1\right)\abs{Q_j - Q} \\
   &\lesssim \left(|Q_j - Q|^3 + 2 |Q|^3 + 1\right)\abs{Q_j - Q} \\
   &\lesssim \left(|Q|^3 + 1\right)\left(|Q| + 1\right) \lesssim |Q|^4 + 1 \in L^1(\partial\mcP_\eps).
  \end{split}
 \]
 Thanks to the continuity of the trace
 operator $H^1(\Omega_\eps)\to H^{1/2}(\partial\mcP_\eps)$
 and to the compact Sobolev embedding 
 $H^{1/2}(\partial\mcP_\eps)\hookrightarrow L^2(\partial\mcP_\eps)$,
 we have $Q_j\to Q$ a.e. on~$\partial\mcP_\eps$ 
 as~$j\to+\infty$, at least along a subsequence.
 Therefore, we can apply Lebesgue's dominated convergence
 theorem to deduce that
 \begin{equation} \label{lsc3}
  \int_{A_j} \abs{f_s(Q_j, \, \nu) - f_s(Q, \nu)} \d\sigma 
  \to 0 \qquad \textrm{as } j\to +\infty.
 \end{equation}
 Finally, we consider the set~$B_j$, where the oppositive inequality
 $\abs{Q}+1 \leq \abs{Q_j-Q}$ holds. We apply~\eqref{hp:fs} 
 again and deduce that, a.e. on~$B_j$, there holds
 \[
  \begin{split}
   \abs{f_s(Q_j, \, \nu) - f_s(Q, \, \nu)} &\lesssim \left(|Q_j|^3 + |Q|^3 + 1\right)\abs{Q_j - Q} \\
   &\lesssim \left(|Q_j - Q|^3 + 2 |Q|^3 + 1\right)\abs{Q_j - Q} \\
   &\lesssim \abs{Q_j - Q}^4 + \left(\abs{Q} + 1\right)^3 \abs{Q_j - Q} \\
   &\lesssim \abs{Q_j - Q}^4.
  \end{split}
 \]
 With the help of Lemma~\ref{lem:trace_eps}, we obtain
 \[
  \begin{split}
   &\int_{B_j} \abs{f_s(Q_j, \, \nu) - f_s(Q, \nu)} \d\sigma 
   \lesssim \int_{\partial\mcP_\eps} \abs{Q_j - Q}^4 \d\sigma \\
   &\qquad \lesssim \int_{\Omega_\eps} \left(\abs{\nabla Q_j - \nabla Q}^2 + \abs{Q_j - Q}^6\right)\d x
   + \eps^{2\alpha - 3} \int_{\Omega_\eps} \abs{Q_j - Q}^4 \d x.
  \end{split}
 \]
 The last term in the right-hand side converges to zero as~$j\to+\infty$,
 due to the compact Sobolev embedding $H^1(\Omega_\eps)\hookrightarrow L^4(\Omega_\eps)$.
 To estimate the integral of~$|Q_j - Q|^6$, we apply the Sobolev embedding
 $H^1_0(\Omega)\hookrightarrow L^6(\Omega)$ to the harmonic extension 
 $E_\eps(Q_j - Q)$, and use Lemma~\ref{lem:extension}. 
 (We have $E_\eps(Q_j - Q)\in H^1_0(\Omega)$ because both~$Q_j$ and~$Q$ 
 are equal to~$g$ on~$\partial\Omega$.) This yields
 \[
  \begin{split}
   &\int_{B_j} \abs{f_s(Q_j, \, \nu) - f_s(Q, \nu)} \d\sigma \\
   &\qquad\qquad \lesssim \int_{\Omega_\eps} \abs{\nabla Q_j - \nabla Q}^2 \d x
    + \left(\int_{\Omega_\eps} \abs{\nabla Q_j - \nabla Q}^2 \d x\right)^3
    + \mathrm{o}(1).
  \end{split}
 \]
 We use the bound~\eqref{hp:bound} to further estimate the right-hand side:
 \[
  \begin{split}
   \int_{B_j} \abs{f_s(Q_j, \, \nu) - f_s(Q, \nu)} \d\sigma 
   \lesssim \left(1 + 16 M^4\right) 
    \int_{\Omega_\eps} \abs{\nabla Q_j - \nabla Q}^2 \d x + \mathrm{o}(1).
  \end{split}
 \]
 By combining this inequality with~\eqref{lsc3}, we obtain that
 \begin{equation} \label{lsc4}
  \liminf_{j\to+\infty} J_\eps[Q_j] - J_\eps[Q] \geq 
  - C_M \eps^{3-2\alpha} \limsup_{j\to+\infty}
  \int_{\Omega_\eps} \abs{\nabla Q_j - \nabla Q}^2 \d x
 \end{equation}
 for some constant~$C_M>0$ that depends on~$M$, but not on~$\eps$.
 The weak convergence $Q_j\rightharpoonup Q$ and~\eqref{lsc0} imply that
 \[
  \int_{\Omega_\eps} \abs{\nabla Q_j - \nabla Q}^2 \d x = 
  \int_{\Omega_\eps} \left(\abs{\nabla Q_j}^2 
  - 2\nabla Q_j:\nabla Q + \abs{\nabla Q}^2 \right)\d x  \to\mu 
 \]
 as $j\to+\infty$. Therefore, by combining \eqref{lsc1}, 
 \eqref{lsc2} and \eqref{lsc4}, we obtain
 \[
  \liminf_{j\to+\infty} \mcF_\eps[Q_j] - \mcF_\eps[Q]  \geq
  \left(\theta - C_M \eps^{3-2\alpha}\right) \mu
 \]
 and the right-hand side is non-negative when~$\eps$ is sufficiently small,
 because~$\alpha < 3/2$.
\end{proof}

We conclude this section with an example of a scalar functional that fails
to be sequentially weakly lower semi-continuous in~$H^1(\Omega_\eps)$.
We assume this time that the inclusions are cubes, i.e.
$\mcP_\eps = \cup_{i=1}^{N_\eps} [x_\eps^i - \eps^\alpha, \, x_\eps^i + \eps^\alpha]^3$.
For~$u\in H^1(\Omega_\eps, \, \R)$, let
\[
 {\mycal G}_\eps(u) := \int_{\Omega_\eps} \left(|\nabla u|^2 + k |u|^q\right)\d x
 - \delta\eps^{3-2\alpha} \int_{\partial\mcP_\eps} |u|^4\,\d\sigma,
\]
where~$k$, $\delta$, $\alpha$, $q$ are positive parameters.

\begin{lemma} \label{lem:not_lsc}
 For any~$k>0$, $\delta>0$, $q<6$, $\alpha>1$ and $\varepsilon>0$,
 ${\mycal G}_\eps$ is not sequentially weakly lower semi-continuous on
 $\{u\in H^1(\Omega_\eps)\colon \ u=0 \textrm{ on } \partial\Omega\}$.
\end{lemma}
\begin{proof}
 Let~$\varphi\in C^\infty_{\mathrm{c}}(B_1)$ be a non-negative function
 such that $\varphi(0) = 1$. Let $y_\eps := x_\eps^1 + (\eps^\alpha, \, 0, \, 0)$.
 The point~$y_\eps$ is the centre of one of the faces of the cube 
 $[x_\eps^1 - \eps^\alpha, \, x_\eps^1 + \eps^\alpha]^3$.
 For~$j\in\N$ and for a fixed~$M>0$, we define
 \[
  u_j(x) := Mj^{1/2}  \varphi(j(x - y_\eps)) \qquad 
  \textrm{if } x\in B(y_\eps, \, 1/j)
 \]
 and $u_j(x) := 0$ otherwise. For~$j$ large enough we have 
 $u_j = 0$ on~$\partial\Omega$ and, by a change of variable, 
 we compute that
 \begin{gather*}
  \int_{\Omega_\eps} \abs{\nabla u_j}^2 \d x 
  = M^2 \int_{B_1^+} \abs{\nabla\varphi}^2 \d x, \quad
  \int_{\Omega_\eps} \abs{u_j}^q \d x
  = M^q j^{q/2 - 3} \int_{B_1^+} \abs{\varphi}^q \d x \to 0
 \end{gather*}
 as~$j\to +\infty$. Here~$B_1^+ := B_1\cap ([0, \, +\infty)\times\R^2)$.
 In particular, $u_j\rightharpoonup 0$ weakly in~$H^1(\Omega_\eps)$.
 However, we have 
 \begin{gather*}
  \int_{\partial\mcP_\eps} \abs{u_j}^4 \d \sigma
  = M^4 \int_{B_1^0} \abs{\varphi}^4 \d\sigma
 \end{gather*}
 where~$B_1^0 := B_1\cap (\{0\}\times\R^2)$, and hence
 \[
  \liminf_{j\to+\infty} {\mycal G}_\eps(u_j) \leq 
  M^2 \int_{B_1^+} \abs{\nabla\varphi}^2 \d x - \delta\eps^{3-2\alpha} M^4 
  \int_{B_1^0} \abs{\varphi}^4 \d\sigma.
 \]
 Now, for any positive value of~$\delta$ and~$\eps$, the right-hand side can 
 be made strictly negative by taking $M$ large enough.
\end{proof}
\begin{remark}[the subcritical case] \label{remark:subcritical}
 The lack of lower semi-continuity for the surface energy
 only arise when the surface energy density~$f_s$ has quartic growth.
 If~$f_s$ has subcritical growth, that is, if there 
 exist~$2\leq p < 4$ and a constant~$C>0$ such that
 \begin{equation} \label{subcritical}
  \abs{f_s(Q, \, \nu)} \leq C\left(|Q|^p + 1\right) 
  \qquad \textrm{for any } (Q, \, \nu)\in\mcS_0\times\mathbb{S}^2,
 \end{equation}
 then the compact Sobolev embedding
 $H^{1/2}(\partial\mcP_\eps)\hookrightarrow L^p(\partial\mcP_\eps)$
 immediately implies the sequential lower semi-continuity of~$\mcF_\eps$
 with respect to the weak topology of~$H^1(\Omega_\eps)$, for any~$\eps$ and~$\alpha$.
 Under the assumption~\eqref{subcritical}, equicoercivity
 for the family of functionals~$(\mcF_\eps)_{\eps>0}$
 (i.e., a statement analogous to that of Proposition~\ref{prop:coercivity})
 can be deduced from Lemma~\ref{lem:trace}, provided 
 that $1 < \alpha < 3/2$ and that
 \[
  f_b(Q) \geq \mu|Q|^{2p-2} - C
 \]
 for some positive constants~$\mu$, $C$ and any~$Q\in\mcS_0$.
The exponent~$2p - 2$ is sharp, as demonstrated by Lemma~\ref{lem:unbounded};
Lemma~\ref{lem:unbounded3} shows that equicoercivity may also be lost when~$\alpha> 3/2$.
\end{remark}

\section{Convergence of local minimisers}
\label{Sec:conv}

\subsection{Pointwise convergence of the surface integral}

The aim of this section is to prove the following result.

\begin{prop} \label{prop:formal}
 Suppose that the assumptions~\eqref{hp:first}--\eqref{hp:last} 
 are satisfied. Then, for any \emph{bounded, Lipschitz} map
 $Q\colon\overline{\Omega}\to\mcS_0$, there holds
 $J_\eps[Q]\to J_0[Q]$ as~$\eps\to 0$.
\end{prop}

Before we give the proof of this result, we set some notation. 
Let~$\Psi\colon\mcS_0\times\SO(3)\to\R$ be the function defined by
\begin{equation} \label{Psi}
 \Psi(Q, \, R) := \int_{\partial\mcP} f_s(Q, \, R\,\nu_{\mcP}) \, \d\sigma
\end{equation}
for any~$(Q, \, R)\in\mcS_0\times\SO(3)$,
where~$\nu_{\mcP}$ denotes the \emph{inward}-pointing unit normal to~$\partial\mcP$.
Since $f_s$ is continuous 
and satisfies $\abs{f_s(Q, \, \nu)}\lesssim\abs{Q}^4 + 1$
as a consequence of~\eqref{hp:fs},
it is readily checked that
\begin{equation} \label{Psi-cont}
 \Psi \textrm{ is continuous and } 
 \abs{\Psi(Q, \, R)} \lesssim \abs{Q}^4 + 1
 \textrm{ for any } (Q, \, R)\in\mcS_0\times\SO(3).
\end{equation}
Let us also consider the sequence of measures
\begin{equation} \label{mu_eps}
 \mu_\eps := \eps^3\sum_{i=1}^{N_\eps}\delta_{x_\eps^i}.
\end{equation}
By assumption~\eqref{hp:conv}, 
$\mu_\eps\rightharpoonup^*\d x\mres\Omega$ as~$\eps\to 0$,
in the sense of measures on~$\R^3$. (Here~$\d x\mres\Omega$ denotes
the restriction of the Lebesgue measure to~$\Omega$.)

In terms of~$\Psi$, 
the function~$f_{hom}\colon\mcS_0\times\overline{\Omega}\to\R$
defined by~\eqref{f_hom} can be equivalently
expressed as
\begin{equation} \label{Psi-f_hom}
 f_{hom}(Q, \, x) = \Psi(Q, \, R_*(x))
 \qquad \textrm{for any } (Q, \, x)\in\mcS_0\times\overline{\Omega},
\end{equation}
where~$R_*\colon\overline{\Omega}\to\SO(3)$ is the continuous map
given by~\eqref{hp:R}. Thanks to~\eqref{Psi-cont} and the
continuity of~$R_*$, we deduce that
\begin{equation} \label{f_hom-cont}
 f_{hom} \textrm{ is continuous and } 
 \abs{f_{hom}(Q, \, x)} \lesssim \abs{Q}^4 + 1
 \textrm{ for any } (Q, \, x)\in\mcS_0\times\overline{\Omega}.
\end{equation}

\begin{proof}[Proof of Proposition~\ref{prop:formal}]
 Let us fix a bounded, Lipschitz map~$Q\colon\overline{\Omega}\to\mcS_0$.
 In order to avoid problems at the boundary, 
 we extend $Q$ to a map $\R^3\to\mcS_0$,
 still denoted~$Q$, which is compactly supported, bounded and Lipschitz. 
 Likewise, we extend~$R_*$ to a continuous 
 map~$\R^3\to\SO(3)$, still denoted~$R_*$.
 Now, let us consider the quantity
 \begin{equation*} 
  \tilde{J}_\eps[Q] := \eps^{3-2\alpha} \sum_{i=1}^{N_\eps}
  \int_{\partial\mcP_\eps^i} f_s(Q(x_\eps^i), \, \nu)\, \d\sigma.
 \end{equation*}
 Since $\nu(x) = R_\eps^i \, \nu_{\mcP}(\eps^{-\alpha}(R_\eps^i)^{\mathsf{T}}(x-x_\eps^i))$
 for any~$x\in\partial\mcP_\eps^i$, by a change of variable we obtain
 \[
  \begin{split}
   \tilde{J}_\eps[Q] &= \eps^3 \sum_{i=1}^{N_\eps}
   \int_{\partial\mcP} f_s(Q(x_\eps^i), \, R_\eps^i\nu_{\mcP})\, \d\sigma \\
   &\stackrel{\eqref{hp:R}}{=} 
   \eps^3 \sum_{i=1}^{N_\eps} \int_{\partial\mcP} 
     f_s(Q(x_\eps^i), \, R_*(x_\eps^i)\nu_{\mcP})\, \d\sigma \\
   &\stackrel{\eqref{Psi}}{=} \eps^3 \sum_{i=1}^{N_\eps} 
     \Psi(Q(x_\eps^i), \, R_*(x_\eps^i))
   \stackrel{\eqref{mu_eps}}{=} 
    \int_{\R^3} \Psi(Q(x), \, R_*(x)) \, \d\mu_\eps(x).
  \end{split}
 \]
 We have $\mu_\eps\rightharpoonup^*\d x\mres\Omega$ 
 by~\eqref{hp:conv}; moreover, $Q$ is continuous by assumption,
 and the functions~$\Psi$, $R_*$ are continuous by~\eqref{Psi-cont},
 \eqref{hp:R} respectively. Therefore, as~$\eps\to 0$ we deduce that
 \begin{equation} 
  \tilde{J}_\eps[Q] \to 
  \int_{\Omega} \Psi(Q(x), \, R_*(x)) \, \d x
  \stackrel{\eqref{Psi-f_hom}}{=} 
  \int_\Omega f_{hom}(Q(x), \, x)\, \d x = J_0[Q]. 
  \label{formal1}
 \end{equation}
 Thus, it suffices to show that 
 $\tilde{J}_\eps[Q] - J_\eps[Q]\to 0$ as~$\eps\to 0$. 
 By applying the assumption~\eqref{hp:fs} we obtain
 \[
  \begin{split}
   \abs{\tilde{J}_\eps[Q] - J_\eps[Q]}
   &\lesssim \eps^{3-2\alpha} 
   \sum_{i=1}^{N_\eps} \int_{\partial\mcP_\eps^i}
   \abs{f_s(Q(x), \, \nu) - f_s(Q(x_\eps^i), \, \nu)} \d \sigma(x) \\
   &\lesssim \eps^{3-2\alpha} 
   \sum_{i=1}^{N_\eps} \int_{\partial\mcP_\eps^i}
   \left(\abs{Q(x)}^3 + \abs{Q(x_\eps^i)}^3 + 1\right)
   \abs{Q(x) - Q(x_\eps^i)} \d \sigma(x) \\
   &\lesssim \eps^{3-2\alpha}
   \left(\norm{Q}_{L^\infty(\R^3)}^3 + 1\right) \mathrm{Lip}(Q) 
   \, \sum_{i=1}^{N_\eps}
   \mathrm{diam}(\mcP_\eps^i) \, \sigma(\partial\mcP_\eps^i).
  \end{split}
 \]
 We have denoted by~$\mathrm{Lip}(Q)$ the Lipschitz constant of~$Q$,
 and by~$\mathrm{diam}(\mcP_\eps^i)$ the diameter of~$\mcP_\eps^i$.
 Finally, we remark that $\mathrm{diam}(\mcP_\eps^i)\lesssim\eps^\alpha$,
 $\sigma(\mcP_\eps^i)\lesssim\eps^{2\alpha}$ for any~$i$,
 and that~$N_\eps\lesssim\eps^{-3}$; therefore, we deduce that
 \[
  \abs{\tilde{J}_\eps[Q] - J_\eps[Q]} \lesssim 
   \left(\norm{Q}_{L^\infty(\R^3)}^3 + 1\right) \mathrm{Lip}(Q) \, \eps^\alpha.
 \]
 By combining this inequality with~\eqref{formal1}, the proposition follows.
\end{proof}

\subsection{Compactness and lower bounds}

The aim of this section is to prove the following

\begin{prop} \label{prop:liminf}
 Suppose that the assumptions~\eqref{hp:first}--\eqref{hp:last} 
 are satisfied. Let $Q_\eps\in H^1_g(\Omega_\eps, \, \mcS_0)$
 be such that $E_\eps Q_\eps\rightharpoonup Q$ weakly in~$H^1(\Omega)$
 as~$\eps\to 0$. Then, there holds
 \[
  \liminf_{\eps\to 0} \mcF_\eps [Q_\eps] \geq \mcF_0[Q], \qquad
  \lim_{\eps\to 0} J_\eps[Q_\eps] = J_0[Q].
 \]
\end{prop}
\begin{remark} \label{remark:Gamma}
 The result above implies that $\mcF_\eps$ $\Gamma$-converges to~$\mcF_0$
 as~$\eps\to 0$, with respect to the weak $H^1$-topology.
 Indeed, Proposition~\ref{prop:liminf} immediately gives that
 $\mcF_0\leq\Gamma\textrm{-}\liminf_{\eps\to 0}\mcF_\eps$.
 To show that $\mcF_0\geq\Gamma\textrm{-}\limsup_{\eps\to 0}\mcF_\eps$,
 take $Q\in H^1_g(\Omega, \, \mcS_0)$ and observe that, 
 since $|\mcP_\eps|\to 0$, Lebesgue's dominated convergence theorem implies 
 $\mcF_\eps[Q] - J_\eps[Q] \to \mcF_0[Q] - J_0[Q]$, while 
 $J_\eps[Q] \to J_0[Q]$ by Proposition~\ref{prop:liminf}.
 Therefore, the constant sequence $Q_\eps := Q$ is a recovery sequence.
\end{remark}

We start by proving an ``approximate continuity'' property of~$J_\eps$.
Note that, thanks to the extension operator~$E_\eps$, we can assume without 
loss of generality that the maps we are working with are defined over the whole of~$\Omega$. 

\begin{lemma} \label{lem:J_cont}
 Suppose that the assumption~\eqref{hp:fs} is satisfied.
 Let~$Q_1$, $Q_2\in H^1_g(\Omega, \mcS_0)$ 
 be such that
 \begin{equation} \label{eq:H1bd}
  \max\left\{\|\nabla Q_1\|_{L^2(\Omega)}, \, \|\nabla Q_2\|_{L^2(\Omega)}\right\} \leq M
 \end{equation}
 for some ($\eps$-independent) constant~$M$. Then, there holds
 \[
  \abs{J_\eps[Q_2] - J_\eps[Q_1]} \leq C \left(\eps^{3/4-\alpha/2}
  + \norm{Q_2 - Q_1}_{L^4(\Omega)}\right)
 \]
 for some~$C>0$ depending only on~$M$, $f_s$, $\Omega$, $\mcP$ and~$g$.
\end{lemma}
\begin{proof}
 By applying the assumption~\eqref{hp:fs} and H\"older inequality, we obtain
 \[
  \begin{split}
   &\abs{J_\eps[Q_2] - J_\eps[Q_1]} \leq \eps^{3-2\alpha} 
   \int_{\partial\mcP_\eps} \left(|Q_1|^3 + |Q_2|^3 + 1\right) \abs{Q_2 - Q_1} \d\sigma \\
   &\qquad \lesssim \underbrace{\left(\eps^{3-2\alpha} \int_{\partial\mcP_\eps} \left(|Q_1|^4 + |Q_2|^4 + 1\right)
   \d\sigma \right)^{3/4}}_{=: I_1} \, \underbrace{\left(\eps^{3-2\alpha} 
   \int_{\partial\mcP_\eps} \abs{Q_2 - Q_1}^4 \d\sigma\right)^{1/4}}_{=: I_2}
  \end{split}
 \]
 We first consider~$I_1$. We apply Lemma~\ref{lem:trace_eps}, and use the fact that
 the total surface area of~$\mcP_\eps$ is of order~$\eps^{2\alpha-3}$:
 \[
  I_1^{4/3} \lesssim \eps^{3-2\alpha} \int_{\Omega_\eps}\left(\abs{\nabla Q_1}^2 +
  \abs{\nabla Q_2}^2 + \abs{Q_1}^6 + \abs{Q_2}^6\right) \d x 
  + \int_{\Omega_\eps}\left(\abs{Q_1}^4 + \abs{Q_2}^4\right)\d x + 1.
 \]
 The right-hand side is bounded in terms of~$M$, due to the $H^1$-bound 
 \eqref{eq:H1bd} and the Sobolev embedding
 $H^1_g(\Omega)\hookrightarrow L^6(\Omega)$;
 therefore, $I_1$ is bounded. Now, we apply Lemma~\ref{lem:trace_eps} to~$I_2$:
 \[
  I_2^{4} \lesssim \eps^{3-2\alpha} \int_{\Omega_\eps}\left(\abs{\nabla Q_2 - \nabla Q_1}^2 +
  \abs{Q_2 - Q_1}^6 \right) \d x 
  + \int_{\Omega_\eps}\abs{Q_2 - Q_1}^4\d x .
 \]
 Again, the first integral in the right hand-side
 is bounded due to~\eqref{eq:H1bd} and Sobolev embeddings,
 so the lemma follows.
\end{proof}

\begin{lemma} \label{lem:J_conv}
 For \emph{any}~$Q\in H_g^1(\Omega, \mcS_0)$, there holds
 $J_\eps[Q]\to J_0[Q]$ as~$\eps\to 0$.
\end{lemma}
\begin{proof}
 Let~$(Q_j)_{j\in\N}$ be a sequence of smooth maps that converge 
 $H^1$-strongly to~$Q$. By the triangle inequality, we have
 \begin{equation} \label{eq:triangle}
  \abs{J_\eps[Q] - J_0[Q]} \leq \abs{J_\eps[Q] - J_\eps[Q_j]} +
  \abs{J_\eps[Q_j] - J_0[Q_j]} + \abs{J_0[Q_j] - J_0[Q]}
 \end{equation}
 To deal with the first term, we apply Lemma~\ref{lem:J_cont} and deduce that
 \[
  \abs{J_\eps[Q] - J_\eps[Q_j]} \lesssim \eps^{3/4-\alpha/2}
  + \norm{Q_j - Q}_{L^4(\Omega)}.
 \]
 The second term in the right-hand side of~\eqref{eq:triangle} 
 tends to zero as~$\eps\to 0$, by Proposition~\ref{prop:formal}.
 We consider now the third term. By the Sobolev embedding 
 $H^1(\Omega)\hookrightarrow L^4(\Omega)$, and up to extraction of
 a subsequence, we can assume that $Q_j\to Q$ a.e. and there exists 
 a function~$\varphi\in L^1(\Omega)$ such that $|Q_j|^4 \leq\varphi$ for any~$j$.
 Since $f_{hom}$ is continuous and $|f_{hom}(Q)|\lesssim |Q|^4 + 1$
 (see~\eqref{f_hom-cont}),
 we can apply Lebesgue's dominated convergence theorem to 
 conclude that $J_0[Q_j] \to J_0[Q]$ as~$j\to +\infty$.
 Putting all this together, we obtain
 \[
  \limsup_{\eps\to 0} \abs{J_\eps[Q] - J_0[Q]}
  \leq \norm{Q_j - Q}_{L^4(\Omega)} + \abs{J_0[Q_j] - J_0[Q]} 
  \to 0 \quad \textrm{as } j\to+\infty,
 \]
 so the lemma follows.
\end{proof}

\begin{proof}[Proof of Proposition~\ref{prop:liminf}]
 Let~$Q_\eps\in H^1_g(\Omega_\eps, \mcS_0)$ be such that 
 $E_\eps Q_\eps\rightharpoonup Q$
 weakly in $H^1(\Omega, \, \mcS_0)$. In particular, $(E_\eps Q_\eps)_{\eps>0}$
 is a bounded sequence in~$H^1(\Omega)$. By the compact Sobolev embedding
 $H^1(\Omega)\hookrightarrow L^4(\Omega)$, we can also assume that 
 $E_\eps Q_\eps \to Q$ strongly in~$L^4$ and a.e. (possibly by taking a 
 non-relabelled subsequence).
 
 We first consider the elastic contribution. The convexity of~$f_e$,
 together with H\"older inequality, implies
 \[
  \begin{split}
   &\int_{\Omega_\eps} 
     \left(f_e(\nabla Q_\eps) - f_e(\nabla Q)\right)\d x \geq 
   \int_{\Omega_\eps} (\nabla f_e)(\nabla Q):(\nabla Q_\eps - \nabla Q)\, \d x \\
   &\geq \int_{\Omega}\!(\nabla f_e)(\nabla Q)\!:\!(\nabla (E_\eps Q_\eps) - \nabla Q)\, \d x 
   - \norm{(\nabla f_e)(\nabla Q)}_{L^2(\mcP_\eps)}\!\norm{\nabla (E_\eps Q_\eps) - \nabla Q}_{L^2(\mcP_\eps)}
  \end{split}
 \]
 Since~$|(\nabla f_e)(\nabla Q)|\lesssim |\nabla Q| + 1$ by assumption~\eqref{hp:fe},
 $(\nabla f_e) (\nabla Q)\in L^2(\Omega)$ indeed. 
 The first term in the right-hand side tends to zero
 as~$\eps\to 0$, because of the $H^1$-weak convergence
 $E_\eps Q_\eps\rightharpoonup Q$. The second term tends 
 to zero as well, because  $(\nabla f_e) (\nabla Q)\in L^2(\Omega)$ and
 $|\mcP_\eps|\to 0$. Then, we deduce
 \begin{equation}\label{eq:liminf1}
  \liminf_{\eps\to 0} \int_{\Omega_\eps} f_e(\nabla Q_\eps)\, \d x \geq 
  \lim_{\eps\to 0} \int_{\Omega_\eps} f_e(\nabla Q)\, \d x = \int_\Omega f_e(\nabla Q)\, \d x.
 \end{equation}

 Now, we consider the integral of the bulk potential~$f_b$. By assumption~\eqref{hp:fb},
 $f_b$ is bounded from below; let~$C$ be a constant such that $f_b + C\geq 0$.
 Since $|\mcP_\eps|\to 0$, the indicator function $\chi_{\Omega_\eps}$
 of~$\Omega_\eps$ converges to~$1$ strongly in~$L^1(\Omega)$
 and hence, up to non-relabelled subsequences, a.e. This fact, together
 with the continuity of~$f_b$, implies that 
 $(f_b(E_\eps Q_\eps) + C)\chi_{\Omega_\eps}\to f_b(Q) + C$ a.e. on~$\Omega$.
 Thus, we can apply Fatou's lemma:
 \begin{equation} \label{eq:liminf2}
  \begin{split}
   \liminf_{\eps\to 0}\int_{\Omega_\eps} f_b(Q_\eps) \, \d x &=
   \liminf_{\eps\to 0}\left(\int_{\Omega} \left(f_b(E_\eps Q_\eps) + C\right) 
      \chi_{\Omega_\eps}\d x - C|\Omega_\eps|\right) \\
   &\geq \int_\Omega (f_b(Q) + C) \, \d x - C|\Omega| 
   = \int_\Omega f_b(Q) \, \d x.
  \end{split}
 \end{equation}

 Finally, it remains to consider the surface integral, $J_\eps$.
 Thanks to Lemma~\ref{lem:J_cont}, we have
 \[
 \begin{split}
  \abs{J_\eps[Q_\eps] - J_0[Q]} &\leq 
  \abs{J_\eps[E_\eps Q_\eps] - J_\eps[Q]} + \abs{J_\eps[Q] - J_0[Q]}  \\
  &\lesssim \eps^{3/4-\alpha/4} + \norm{E_\eps Q_\eps - Q}_{L^4(\Omega)}
  + \abs{J_\eps[Q] - J_0[Q]} .
 \end{split}
 \]
 All the terms in the right-hand side converge to zero as~$\eps\to 0$,
 due to the strong $L^4$-convergence $E_\eps Q_\eps \to Q$
 and to Lemma~\ref{lem:J_conv}. Therefore, we have $J_\eps[Q_\eps]\to J_0[Q]$.
 This fact, combined with~\eqref{eq:liminf1} and~\eqref{eq:liminf2}, 
 concludes the proof of the proposition.
\end{proof}

\subsection{Proof of Theorems~\ref{th:loc-min} and~\ref{th:global-min}}

The aim of this section is to complete the proof of the main results, 
Theorems~\ref{th:loc-min} and~\ref{th:global-min}.
Let~$Q_0\in H^1_g(\Omega, \mcS_0)$ be an isolated $H^1$-local minimiser
for $\mcF_0$ --- that is, suppose that there exists $\delta_0>0$ such that
\begin{equation} \label{eq:loc-min}
 \mcF_0[Q_0] < \mcF_0[Q] \quad \textrm{for all } Q\in H^1_g(\Omega, \mcS_0)
  \ \textrm{ such that } 0 < \norm{Q - Q_0}_{H^1(\Omega)} \leq\delta_0.
\end{equation}
Let
\[
 \mcB_\eps := \{Q\in H^1_g(\Omega_\eps,\mcS_0)\colon
 \norm{E_\eps Q-Q_0}_{H^1(\Omega)}\leq\delta_0\} \! .
\]
Thanks to Proposition~\ref{prop:lsc}, there exists~$\eps_0> 0$
such that, for any~$0 < \eps \leq \eps_0$, the functional~$\mcF_\eps$
is sequentially lower semi-continuous on~$\mcB_\eps$, with respect 
to the weak $H^1$-topology. Therefore, $\mcF_\eps$ admits a 
minimiser~$Q_\eps$ on~$\mcB_\eps$.

\begin{prop} \label{prop:convergence}
 There holds $E_\eps Q_\eps\to Q_0$ 
 \emph{strongly} in~$H^1_g(\Omega, \, \mcS_0)$ as $\eps\to 0$.
\end{prop}

Because the convergence is strong in~$H^1$, for small~$\varepsilon$
the map~$Q_\eps$ lies in the interior of~$\mcB_\eps$ and, in particular,
it is an $H^1$-local minimiser of~$\mcF_\eps$.
Therefore, Theorem~\ref{th:loc-min} follows immediately 
from Proposition~\ref{prop:convergence}.

\begin{proof}[Proof of Proposition~\ref{prop:convergence}]
 Let~$\mcB_0$ be the set of maps $Q\in H^1_g(\Omega,\mcS_0)$ such that
 $\|Q-Q_0\|_{H^1(\Omega)}\leq\delta_0$.
 Because~$Q_\eps\in\mcB_\eps$, we can extract a (non-relabelled)
 subsequence so that $E_\eps Q_\eps$ converges $H^1$-weakly to
 $\overline{Q}\in\mcB_0$. Proposition~\ref{prop:liminf} 
 implies that $\overline{Q}$ is a minimiser of~$\mcF_0$ in~$\mcB_0$.
 Indeed, we certainly have ${Q_0}_{|\Omega_\eps}\in\mcB_\eps$ for~$\eps$ small enough, 
 and the minimality of~$Q_\eps$ implies that
 $\mcF_\eps[Q_\eps] \leq \mcF_\eps[Q_0]$. Therefore,
 by applying Proposition~\ref{prop:liminf}, 
 we have
 \begin{equation} \label{eq:Gamma}
  \mcF_0[\overline{Q}] \leq \liminf_{\eps\to 0} \mcF_\eps[Q_\eps]
  \leq \limsup_{\eps\to 0} \mcF_\eps[Q_\eps]
  \leq \lim_{\eps\to 0} \mcF_\eps[Q_0] = \mcF_0[Q_0].
 \end{equation}
 Due to~\eqref{eq:loc-min}, we must have $\overline{Q} = Q_0$.
 The uniqueness of the limit (and the fact that weak convergence 
 on bounded subsets of~$H^1$ is metrisable) implies that the whole sequence 
 $(Q_\eps)_{0<\eps\leq\eps_0}$ converges weakly to~$Q_0$.

 Now, it only remains to show that the convergence is actually strong.
 By Assumption~\eqref{hp:fe}, there exists $\theta>0$ such that
 the function $\tilde{f}_e(D) := f_e(D) - \theta|D|^2$
 is convex. Therefore, we can repeat the arguments 
 for~\eqref{eq:liminf1}--\eqref{eq:liminf2} in the proof of Proposition~\ref{prop:liminf}
 and prove that
 \begin{gather*}
  \liminf_{\eps\to 0} \int_{\Omega_\eps} \tilde{f}_e(\nabla Q_\eps)\, \d x
  \geq \int_{\Omega} \tilde{f}_e(\nabla Q_0)\,\d x, \quad
  \theta\liminf_{\eps\to 0} \int_{\Omega_\eps} \abs{\nabla Q_\eps}^2 \d x
  \geq \theta \int_{\Omega} \abs{\nabla Q_0}^2 \d x \\  
  \liminf_{\eps\to 0} \int_{\Omega_\eps} f_b(Q_\eps)\, \d x
  \geq \int_{\Omega} f_b(Q_0)\,\d x
 \end{gather*}
 We know that $J_\eps[Q_\eps]\to J_0[Q_0]$ by Proposition~\ref{prop:liminf}.
 At the same time, as a byproduct of~\eqref{eq:Gamma}, 
 we see that $\mcF_\eps[Q_\eps]\to\mcF_0[Q]$.
 Therefore, there must hold
 \begin{equation*} 
  \theta \lim_{\eps\to 0}\int_{\Omega_\eps} \abs{\nabla Q_\eps}^2 \d x
  = \theta \int_{\Omega} \abs{\nabla Q_0}^2 \d x.
 \end{equation*}
 This implies that $\nabla (E_\eps Q_\eps)\chi_{\Omega_\eps} \to \nabla Q$ strongly in~$L^2(\Omega)$,
 where~$\chi_{\Omega_\eps}$ is the characteristic function of~$\Omega_\eps$.
 On the other hand, in the proof of Lemma~\ref{lem:extension}
 (see Eq.~\eqref{harm_extension}) we have shown that
 \[
  \norm{\nabla (E_\eps Q_\eps)}_{L^2(\mcP_\eps)} \lesssim 
  \norm{\nabla Q_\eps}_{L^2(\widetilde{\mcP_\eps}\setminus\mcP_\eps)} \! ,
 \]
 where $\widetilde{\mcP}_\eps := \cup_i (x^i_\eps + 2\eps^\alpha R^i_\eps \mcP)$.  
 The right-hand side tends to zero as~$\eps\to 0$, because  
 the sequence $(|\nabla (E_\eps Q_\eps)|^2\chi_{\Omega_\eps})_{\eps> 0}$ is strongly 
 compact in~$L^1$ (hence equi-integrable)
 and $|\widetilde{\mcP_\eps}\setminus\mcP_\eps|\to 0$.
 Therefore, we conclude that $\nabla (E_\eps Q_\eps)\to \nabla Q$
 strongly in~$L^2(\Omega)$, and the proposition follows.
\end{proof}

\begin{proof}[Proof of Theorem~\ref{th:global-min}]
 Take a map~$Q\in H^1_g(\Omega, \, \mcS_0)$. By 
 Proposition~\ref{prop:liminf} and Remark~\ref{remark:Gamma},
 we see that
 \[
  \limsup_{\eps\to 0} \inf_{H^1_g(\Omega_\eps, \, \mcS_0)} \mcF_\eps
  \leq \limsup_{\eps\to 0} \mcF_\eps [Q] = \mcF_0[Q].
 \]
 By applying Proposition~\ref{prop:coercivity} we deduce that,
 at least for~$\eps$ small enough, any minimising sequence
 for~$\mcF_\eps$ in~$H^1_g(\Omega, \mcF_\eps)$ is bounded in~$H^1$
 (uniformly with respect to $\eps$). Then, thanks to the lower semi-continuity 
 provided by Proposition~\ref{prop:lsc}, for~$\eps$ small enough the functional
 $\mcF_\eps$ admits a global minimiser in~$H^1_g(\Omega, \mcS_0)$.
 The theorem now follows by the same arguments used in the 
 proof of Proposition~\ref{prop:convergence}.
\end{proof}

\section{Applications to the Landau-de Gennes model}
\label{Sec:LDG}

In this section, we consider the Landau-de Gennes model 
for nematic liquid crystals. In this model, the elastic
energy density is given by
\[
 f_e^{LdG}(\nabla Q) := L_1 \, \partial_k Q_{ij} \, \partial_k Q_{ij}
 + L_2 \, \partial_j Q_{ij} \, \partial_k Q_{ik}
 + L_3 \, \partial_j Q_{ik} \, \partial_k Q_{ij}
\]
(Einstein's summation convention is assumed). We impose the inequalities
\begin{equation} \label{Longa}
 L_1 > 0, \qquad - L_1 < L_3 < 2L_1,
 \qquad -\frac{3}{5} L_1 - \frac{1}{10} L_3 < L_2, 
\end{equation}
which guarantee the strong convexity 
of~$f_e^{LdG}$~\cite{Longa}. The bulk energy
density is a quartic polynomial in the scalar invariants of~$Q$:
\begin{equation*} 
 f_b^{LdG}(Q) := a\,\tr(Q^2) - b\,\tr(Q^3) + c\left(\tr(Q^2)\right)^2 \!.
\end{equation*}
The positive coefficients~$b$, $c$ depend on the material
but not on the temperature. The coefficient~$a\in\R$ does depend 
on the temperature~$T$ and is given by
$a = a_*(T - T_*)$, where~$a_*$ is a material parameter
and $T_*$ is a characteristic temperature of the material
(it is the temperature where the isotropic state loses stability).

Suppose that we are given a nematic host with Landau-de Gennes
coefficients $(a, \, b, \, c)$. We aim to obtain a colloidal suspension
with pre-assigned effective Lan\-dau-de Gen\-nes coefficients 
$(a^\prime, \, b^\prime, \, c^\prime)$.
We assume that the inclusions are spherical, that is, 
$\mcP = \overline{B}_1$ and
\[
 \mcP_\eps = \bigcup_{i=1}^{N_\eps} 
 \overline{B}_{\eps^\alpha}(x^i_\eps).
\]
The centers of the inclusions, $x^i_\eps$, are chosen in such a way
that~\eqref{hp:Omega_eps} and~\eqref{hp:conv} are satisfied
(for instance, we may take the~$x^i_\eps$'s to be periodically distributed).
We wish to choose the surface elastic energy~$f_s$
in such a way that local minimisers of the Landau-de Gennes functional
\begin{equation} \label{Feps-Ldg}
 \begin{split}
 \mcF_\eps[Q] &= \int_{\Omega_\eps} \left(f_e^{LdG}(\nabla Q) 
 + a\,\tr(Q^2) - b\, \tr(Q^3) + c\left(\tr(Q^2)\right)^2 \right)\d x \\
 &\qquad\qquad\qquad + \eps^{3-2\alpha} 
 \int_{\partial\mcP_\eps} f_s(Q, \, \nu) \, \d\sigma
 \end{split}
\end{equation}
converge to local minimisers of the homogenised functional
\begin{equation} \label{Fzero-LdG}
 \mcF_0[Q] = \int_\Omega \left(f_e^{LdG}(\nabla Q) 
 + a^\prime \, \tr(Q^2) - b^\prime \, \tr(Q^3) +
 c^\prime \left(\tr(Q^2)\right)^2 \right) \d x.
\end{equation}
Our choice of~$f_s$ must be consistent with the physical 
symmetries of the system. We impose that $f_s$ is invariant 
with respect to orthogonal transformations, that is,
\begin{equation} \label{invariance-}
  f_s(UQU^{\mathsf{T}}, \, U\nu) = f_s(Q, \, \nu) \qquad 
  \textrm{for any } (Q, \, \nu)\in\mcS_0\times\mathbb{S}^2, \ U\in\O(3).
\end{equation}
This condition combines invariance with respect to rotations
(i.e. frame-in\-dif\-fer\-ence) and invariance with
respect to the orientation of the surface.

\begin{prop} \label{prop:surface_energy-}
 A function~$f_s\colon\mcS_0\times\mathbb{S}^2\to\R$ 
 satisfies~\eqref{invariance-} if and only if there exists a
 function~$\tilde{f}_s\colon\R^4\to\R$ such that
 \[
  f_s(Q, \, \nu) = \tilde{f}_s(\tr(Q^2), \, \tr(Q^3),
  \, \nu\cdot Q\nu, \, \nu\cdot Q^2\nu)
 \]
 for all~$(Q, \, \nu)\in\mcS_0\times\mathbb{S}^2$.
\end{prop}

We postpone the proof of this result to the Appendix~\ref{app:surface}.
We define the surface energy density
\begin{equation} \label{fs-LdG}
 \begin{split}
  f_s(Q, \, \nu) &= \frac{3}{4\pi}(a^\prime - a) (\nu\cdot Q^2\nu)
   + \frac{15}{8\pi}(b^\prime - b) (\nu\cdot Q\nu)(\nu\cdot Q^2\nu) \\
  & \qquad\qquad\qquad +\frac{15}{8\pi}(c^\prime - c) (\nu\cdot Q^2\nu)^2.
 \end{split}
\end{equation}
This choice is consistent with the physical invariance~\eqref{invariance-}.
An expression of this type has  been proposed by Sluckin 
and Poniewierski~\cite{SluckinPon}, based on an idea of Goossens~\cite{Goossens}
(see also~\cite[Eq.~(5)]{genpot}, \cite[Eq.~(9.b)]{Rey} and the references therein).

\begin{theorem} \label{th:conv-LdG}
 Let~$(a, \, b, \, c)$ and~$(a^\prime, \, b^\prime, \, c^\prime)$
 be two set of parameters with~$c>0$, $c^\prime>0$.
 Suppose that the inequalities~\eqref{Longa} are satisfied.
 Then, for any isolated local minimiser~$Q_0$ of the functional~$\mcF_0$
 defined by~\eqref{Fzero-LdG}, and for~$\eps>0$ small enough,
 there exists a local minimiser~$Q_\eps$ of the functional~$\mcF_\eps$,
 defined by~\eqref{Feps-Ldg}, such that $E_\eps Q_\eps\to Q_0$
 strongly in~$H^1(\Omega, \, \mcS_0)$.
\end{theorem}
\begin{proof}
 This is a particular case of our main result,
 Theorem~\ref{th:loc-min}. Indeed, if~\eqref{Longa} holds
 and~$c>0$, $c^\prime>0$, then all the 
 conditions~\eqref{hp:first}--\eqref{hp:last} are satisfied. 
 All we need to do is to compute the homogenised potential, $f_{hom}$,
 defined by~\eqref{f_hom}. Since the inclusions are spheres, 
 we can take the rotation field~$R_*$ (see~\eqref{hp:R})
 to be the identity. Then, for any~$Q\in\mcS_0$ we have
 \[
  \begin{split}
   f_{hom}(Q) &= \int_{\partial B_1} f_s(Q, \, -\nu) \, \d\sigma(\nu)\\
   &= \frac{3}{4\pi}(a^\prime - a) \int_{\partial B_1}(\nu\cdot Q^2\nu) \,\d\nu
   + \frac{15}{8\pi}(b^\prime - b) \int_{\partial B_1}(\nu\cdot Q\nu)(\nu\cdot Q^2\nu)\,\d\nu\\
  & \qquad\qquad\qquad +\frac{15}{8\pi}(c^\prime - c) \int_{\partial B_1}(\nu\cdot Q^2\nu)^2\,\d\nu.
  \end{split}
 \]
 We claim that
 \begin{gather*}
   \int_{\partial B_1}(\nu\cdot Q^2\nu) \,\d\nu 
     = \frac{4\pi}{3}\tr(Q^2), \qquad
   \int_{\partial B_1}(\nu\cdot Q\nu)(\nu\cdot Q^2\nu)\,\d\nu
     = \frac{8\pi}{15}\tr(Q^3) \\
  \int_{\partial B_1}(\nu\cdot Q^2\nu)^2\,\d\nu = \frac{8\pi}{15} (\tr(Q^2))^2.
 \end{gather*}
 The proof of this claim is given in Lemma~\ref{lem:int}, in the appendix.
 Then, we obtain
 \[
  f_{hom}(Q) = (a^\prime - a)\,\tr(Q^2) + (b^\prime - b)\,\tr(Q^3)
  + (c^\prime - c)\,(\tr(Q^2))^2
 \]
 and the theorem follows.
\end{proof}

\begin{remark} An alternative choice of~$f_s$ is
 \begin{equation*}
 \begin{split}
  f_s(Q, \, \nu) &= \frac{15}{8\pi}(a^\prime - a) (\nu\cdot Q\nu)^2
   + \frac{15}{8\pi}(b^\prime - b) (\nu\cdot Q\nu)(\nu\cdot Q^2\nu) \\
  & \qquad\qquad\qquad +\frac{15}{8\pi}(c^\prime - c) (\nu\cdot Q^2\nu)^2.
 \end{split}
 \end{equation*}
 The same argument as above, combined with Lemma~\ref{lem:int} in the appendix, shows
 that Theorem~\ref{th:conv-LdG} holds for this choice of~$f_s$ also.
\end{remark}

\begin{remark}
 In case~$b^\prime = b$, $c^\prime = c$ and~$a^\prime > a$, 
 Theorem~\ref{th:conv-LdG} holds also if we take the Rapini-Papoular-type
 surface energy density defined by
 \begin{equation} \label{Rapini}
 \begin{split}
  f_s(Q, \, \nu) &:= \frac{1}{4\pi}(a^\prime - a) \, \tr(Q - Q_\nu)^2,
 \end{split}
 \end{equation}
 where~$Q_\nu := \nu\otimes\nu - \Id/3$. 
 This energy favours homeotropic anchoring at the boundary
 of the inclusions. We remark that, for any 
 constant~$Q\in\mcS_0$, there holds
 \begin{equation*}
 \begin{split}
  \int_{\partial B_1}\tr(Q - Q_\nu)^2 \, \d\sigma
  =  4\pi\,\tr(Q^2) - 2\tr\left(Q\int_{\partial B_1} Q_\nu \, \d\sigma\right)
  + \int_{\partial B_1} \tr(Q_\nu)^2 \, \d\sigma.
 \end{split}
 \end{equation*}
 The last term in the right-hand side
 integrates to a constant that does not depend on~$Q$, and hence,
 it can be dropped from the energy. The second term in the 
 right-hand side vanishes, because $\int_{\partial B_1}Q_\nu\,\d\sigma$ is, 
 by symmetry reasons, a multiple of the identity
 (see Lemma~\ref{lem:int}) and~$\tr Q=0$.
 Therefore, in case of spherical inclusions, a surface 
 anchoring energy density such as~\eqref{Rapini} 
 produces the same effect as an additional 
 term~$(a^\prime - a)\, \tr(Q^2)$ in the bulk energy density.
\end{remark}

\par\noindent{\bf Acknowledgement.}
The work of both authors is supported by the Basque Government through the BERC 2018-2021
program, by Spanish Ministry of Economy and Competitiveness MINECO through BCAM
Severo Ochoa excellence accreditation SEV-2017-0718 and through project MTM2017-82184-R
funded by (AEI/FEDER, UE) and acronym ``DESFLU''.
\par The authors would like to thank the Isaac Newton Institute for Mathematical Sciences for support and hospitality during the programme {\it ``The design of new materials programme"} when work on this paper was undertaken. This work was supported by: EPSRC grant numbers EP/K032208/1 and EP/R014604/1.
\begin{appendix}
\section{Technical results}

\subsection{Physical symmetries of the surface energy density}
\label{app:surface}

The aim of this section is to prove Proposition~\ref{prop:surface_energy-},
that is, to characterise the surface energies densities that
are invariant by orthogonal transformations. More precisely,
we will prove the following

\begin{prop} \label{prop:surface_energy}
 Let~$f\colon\mcS_0\times\R^3\to\R$ be a function that satisfies
 \begin{equation} \label{invariance}
  f(UQU^{\mathsf{T}}, \, Uu) = f(Q, \, u) \qquad 
  \textrm{for any } (Q, \, u)\in\mcS_0\times\R^3, \ U\in\O(3).
 \end{equation}
 Then, there exists a function~$\tilde{f}\colon\R^5\to\R$ such that
 \[
  f(Q, \, u) = \tilde{f}(\tr(Q^2), \, \tr(Q^3), \, 
  |u|^2, \, u\cdot Qu, \, u\cdot Q^2u)
 \]
 for all~$(Q, \, u)\in\mcS_0\times\R^3$.
\end{prop}

It is straightforward to check that the converse 
also holds, so Proposition~\ref{prop:surface_energy} implies
Proposition~\ref{prop:surface_energy-}.

\begin{lemma}\label{lem:CayleyHamilton}
 Let~$Q\in\mcS_0$ and~$u$, $v\in\R^3$ be such that
 \begin{equation} \label{sameinvariants}
  |u| = |v|,  \quad u\cdot Qu = v\cdot Qv, 
  \quad u\cdot Q^2u = v\cdot Q^2v.
 \end{equation}
 Then, there exists~$U\in\O(3)$ such that~$UQ=QU$ and~$v=Uu$.
\end{lemma}
\begin{proof}
 Let $\lambda_1$, $\lambda_2$, $\lambda_3$ denote the eigenvalues of~$Q$,
 ordered in such a way that $|\lambda_1|\leq |\lambda_2|\leq |\lambda_3|$.
 If two of the eigenvalues are equal to zero, then~$Q=0$ (because~$\tr Q=0$).
 On the other hand, the assumption $|u| = |v|$
 implies that we can write~$v = Uu$ for some~$U\in\SO(3)$, so if~$Q=0$
 there is nothing to prove.
 Therefore, we can assume that $\lambda_2\neq 0$, $\lambda_3\neq 0$.
 
 Let~$R\in\O(3)$ be such that $RQR^{\mathsf{T}} = \mathrm{diag}(\lambda_1, \, \lambda_2, \, \lambda_3)$.
 We define~$\bar{u}:=Ru$, $\bar{v}:=Rv$.
 By Cayley-Hamilton theorem, we know that $Q^3 = \alpha_1Q +\alpha_0\Id$
 for some numbers~$\alpha_0$, $\alpha_1$
 depending on the eigenvalues of~$Q$. Therefore, by induction, from
 the assumption~\eqref{sameinvariants} we deduce that
 \begin{equation*} 
  R^{\mathsf{T}}\bar{u}\cdot Q^jR^{\mathsf{T}}\bar{u} =
  R^{\mathsf{T}}\bar{v}\cdot Q^jR^{\mathsf{T}}\bar{v} \qquad \textrm{for any } j\in\N
 \end{equation*}
 or, equivalently,
 \begin{equation} \label{sameinvariants2}
  \lambda_1^j \, \bar{u}_1^2 + \lambda_2^j \, \bar{u}_2^2 + \lambda_3^j \, \bar{u}_3^2 = 
  \lambda_1^j \, \bar{v}_1^2 + \lambda_2^j \, \bar{v}_2^2 + \lambda_3^j \, \bar{v}_3^2
  \qquad \textrm{for any } j\in\N.
 \end{equation}
 Now, we distinguish four cases, depending on the eigenvalues of~$Q$.
 
 \medskip
 \noindent
 \textbf{Case~1:} $|\lambda_1| < |\lambda_2| < |\lambda_3|$. 
 In this case, we have $\lambda_1\neq 0$ because otherwise, due to the
 constraint $\tr Q=0$, we would have $\lambda_2 = -\lambda_3$. Let us
 divide both sides of~\eqref{sameinvariants2} by~$\lambda_3^j$ and let~$j\to+\infty$.
 We have $(\lambda_1/\lambda_3)^j\to 0$, $(\lambda_2/\lambda_3)^j\to 0$ 
 and so we conclude that~$\bar{u}_3^2 = \bar{v}_3^2$.
 This argument can be iterated to show that $\bar{u}_2^2 = \bar{v}_2^2$, $\bar{u}_1^2 = \bar{v}_1^2$.
 Therefore, we can write $\bar{v} = \bar{U}\bar{u}$ where~$\bar{U}$ is diagonal matrix that
 satisfies $\bar{U}_{ii}\in\{1, \, - 1\}$ for~$i\in\{1, \, 2, \, 3\}$
 (no summation on~$i$ is implied). Clearly, $\bar{U}\in\O(3)$ and~$\bar{U}$ commutes with~$RQR^{\mathsf{T}}$
 because both matrices are diagonal. Then, the matrix~$U := R^{\mathsf{T}}\bar{U}R$
 has all the desired properties.
 
 \medskip
 \noindent
 \textbf{Case~2:} $|\lambda_1| = |\lambda_2| < |\lambda_3|$.
 Also in this case, $\lambda_1\neq 0$ because we have assumed that $\lambda_2\neq 0$.
 If we had $\lambda_1 = -\lambda_2$ then we would get~$\lambda_3= 0$ because~$\tr Q=0$;
 but we assumed that $\lambda_3\neq 0$, so we must have $\lambda_1 = \lambda_2$.
 Now, as in Case~1, we divide both sides of~\eqref{sameinvariants2} by~$\lambda_3^j$
 and let~$j\to+\infty$, thus obtaining that $\bar{u}_3^2 = \bar{v}_3^2$
 and $\bar{u}_1^2 + \bar{u}_2^2 = \bar{v}_1^2 + \bar{v}_2^2$.
 Then, we can write $\bar{v} = \bar{U}\bar{u}$, where the matrix~$\bar{U}$ has the form
 \[
  \bar{U} = \left(\begin{matrix}
              S & 0     \\ 
              0 & \pm 1 \\
             \end{matrix}\right) \qquad 
             \textrm{for some } \ S\in\SO(2).
 \]
 We have $\bar{U}\in\O(3)$ and $\bar{U}$ commutes with
 $RQR^{\mathsf{T}}= \mathrm{diag}(\lambda_1, \, \lambda_2, \, \lambda_3)$
 because~$\lambda_1=\lambda_2$. Also in this case, the lemma is proved because 
 $U := R^{\mathsf{T}}\bar{U}R$ has all the desired properties.
 
 \medskip
 \noindent
 \textbf{Case~3:} $\lambda_2 = \lambda_3$ and~$|\lambda_1| < |\lambda_2|$.
 Again, the constraint $\tr Q = 0$ implies~$\lambda_1 = -2\lambda_2\neq 0$.
 By dividing both sides of~\eqref{sameinvariants2} by~$\lambda_3^j$
 and letting~$j\to+\infty$, we obtain that
 $\bar{u}_2^2 + \bar{u}_3^2 = \bar{v}_2^2 + \bar{v}_3^2$, then
 $\bar{u}_1^2 = \bar{v}_1^2$. We conclude the proof as in Case~2.
 
 \medskip
 \noindent
 \textbf{Case~4:} $\lambda_2 = -\lambda_3\neq 0$ and~$\lambda_1=0$.
 By writing~\eqref{sameinvariants2} with the choices~$j=1$, $j=2$, we obtain the system
 \[
  \begin{cases}
   \bar{u}_2^2 - \bar{u}_3^2 = \bar{v}_2^2 - \bar{v}_3^2 \\
   \bar{u}_2^2 + \bar{u}_3^2 = \bar{v}_2^2 + \bar{v}_3^2
  \end{cases}
 \]
 whence $\bar{u}_2^2 = \bar{v}_2^2$, $\bar{u}_3^2 = \bar{v}_3^2$.
 Then, using the assumption $|u| = |v|$, we also obtain that $\bar{u}_1^2 = \bar{v}_1^2$
 and we conclude the proof arguing as in Case~1.
 
 \medskip
 In principle, we still need to consider the case
 $\abs{\lambda_1}=\abs{\lambda_2}=\abs{\lambda_3}$.
 However, if we had $\abs{\lambda_1}=\abs{\lambda_2}=\abs{\lambda_3}$
 then the constraint $\tr Q = 0$ would imply $Q = 0$, so there is nothing left to prove.
\end{proof}

\begin{proof}[Proof of Proposition~\ref{prop:surface_energy}]
 Let~$(Q, \, u)$, $(P, \, v)\in\mcS_0\times\R^3$ be such that
 $\tr(Q^2) = \tr(P^2)$, $\tr(Q^3) = \tr(P^3)$ and $u\cdot Q^ju = v\cdot P^jv$
 for~$j\in\{0, \, 1, \, 2\}$. In particular, $Q$ and~$P$ have the same scalar
 invariants, hence the same eigenvalues, and we can write $P = RQR^{\mathsf{T}}$
 for some~$R\in\O(3)$. Then, we have
 \[
  u\cdot Q^ju = v\cdot RQ^jR^{\mathsf{T}} v  = R^{\mathsf{T}} v\cdot Q^jR^{\mathsf{T}} v 
  \qquad \textrm{for } j\in\{0, \, 1, \, 2\}.
 \]
 By applying Lemma~\ref{lem:CayleyHamilton}, we find~$U\in\O(3)$ such that
 $QU = UQ$ and~$R^{\mathsf{T}} v = Uu$. Therefore, there holds
 \[
  f(Q, \, u) \stackrel{\eqref{invariance}}{=} f(UQU^{\mathsf{T}}, \, Uu) = f(Q, \, R^{\mathsf{T}} v)
  \stackrel{\eqref{invariance}}{=} f(RQR^{\mathsf{T}}, \, v) = f(P, \, v).
 \]
 Thus, we can define unambigously a function~$\tilde{f}$ with the required property.
\end{proof}

\subsection{Integrated energy densities}

We will use the notation~$\abs{Q}^2 := \tr(Q^2)$.

\begin{lemma} \label{lem:int}
 For $Q\in\mcS_0$ we have:
 \begin{gather}
  \int_{\mathbb{S}^2} (\nu\cdot Q^2\nu) \, \d\nu 
     = \frac{4\pi}{3} |Q|^2 \label{tr2} \\
  \int_{\mathbb{S}^2} (\nu\cdot Q\nu) (\nu \cdot Q^2\nu) \, \d\nu 
     = \frac{8\pi}{15}\tr(Q^3) \label{tr3} \\
  \int_{\mathbb{S}^2} (\nu \cdot Q^2\nu)^2 \, \d\nu
     = \frac{8\pi}{15}|Q|^4 \label{tr4} \\
  \int_{\mathbb{S}^2} (\nu \cdot Q\nu)^2 \, \d\nu
     = \frac{8\pi}{15}|Q|^2. \label{tr2bis}
 \end{gather}
 More generally, let~$f\colon\mcS_0\times\R^3\to\R$ be a function that satisfies
 \begin{equation} \label{invariance+}
  f(UQU^{\mathsf{T}}, \, Uu) = f(Q, \, u) \qquad 
  \textrm{for any } (Q, \, u)\in\mcS_0\times\R^3, \ U\in\O(3).
 \end{equation}
 Then, there exists a function $h:\R^2\to\R$ such that
 \[
  \int_{\mathbb{S}^2}f(Q, \, \nu)\,\d\nu = h(|Q|^2, \, \tr(Q^3)).
 \]
\end{lemma}

\begin{proof}
For any $Q$ in $\mcS_0$ there exists a diagonal 
matrix $D=\textrm{diag}(\lambda_1,\lambda_2,\lambda_3)$
with $\lambda_1,\lambda_2,\lambda_3$ the eigenvalues of $Q$, hence 

\be
\lambda_1+\lambda_2+\lambda_3=0,
\ee and a rotation matrix $R\in SO(3)$ such that $RQR^{\mathsf{T}}=D$.

\begin{enumerate}
\item Proof of~\eqref{tr2}. There holds
\bea
\int_{\mathbb{S}^2} (\nu\cdot Q^2\nu) \, \d\nu
&=\int_{\mathbb{S}^2} (R^{\mathsf{T}}\nu\cdot Q^2R^{\mathsf{T}}\nu)\,\d\nu\non\\
&=\int_{\mathbb{S}^2} \nu\cdot RQ^2R^{\mathsf{T}}\nu\,\d\nu\non\\
&=\int_{\mathbb{S}^2}\nu\cdot \underbrace{RQR^{\mathsf{T}}}_{=D}\underbrace{RQR^{\mathsf{T}}}_{=D}\nu\,\d\nu\non\\
&=\int_{\mathbb{S}^2} (\lambda_1^2 x_1^2+\lambda_2^2 x_2^2+\lambda_3^2 x_3^2)\,\d x_1\,\d x_2\,\d x_3
\eea
By symmetry considerations we have:
$$
\int_{\mathbb{S}^2}x_1^2\,\d x_1\,\d x_2\,\d x_3
=\int_{\mathbb{S}^2}x_2^2\,\d x_1\,\d x_2\,\d x_3
=\int_{\mathbb{S}^2}x_3^2\,\d x_1\,\d x_2\,\d x_3
$$ 
and since $x_1^2+x_2^2+x_3^2=1$ we have 
$\int_{\mathbb{S}^2}x_1^2\,\d x_1\,\d x_2\,\d x_3
=\frac{1}{3}|\mathbb{S}^2|=\frac{4}{3}\pi$ and 
\be
\int_{\mathbb{S}^2} (\nu\cdot Q^2\nu) \, \d\nu=\frac{4\pi}{3} |Q|^2.
\ee

\item Proof of~\eqref{tr3}. Arguing as before we have:
\begin{align*}
&\int_{\mathbb{S}^2} (\nu\cdot Q\nu) (\nu\cdot Q^2\nu)\,\d\nu
= \int_{\mathbb{S}^2} (\lambda_1x_1^2+\lambda_2x_2^2+\lambda_3x_3^2)
(\lambda_1^2 x_1^2+\lambda_2^2 x_2^2+\lambda_3^2 x_3^2) \,\d x\non\\
&\quad = \int_{\mathbb{S}^2}(\lambda_1^3x_1^4+\lambda_2^3x_2^4+\lambda_3^3x_3^4) \,\d x \non\\
&\quad + \int_{\mathbb{S}^2}\lambda_1\lambda_2(\lambda_1+\lambda_2)x_1^2x_2^2
+ \lambda_1\lambda_3(\lambda_1+\lambda_3)x_1^2x_3^2 
+ \lambda_2\lambda_3(\lambda_2+\lambda_3)x_2^2x_3^2\,\d x\non\\
&\quad = c_{31}\tr(Q^3) + c_{32}\bigg(\lambda_1\lambda_2(\lambda_1+\lambda_2)
+ \lambda_1\lambda_3(\lambda_1+\lambda_3)
+ \lambda_2\lambda_3(\lambda_2+\lambda_3)\bigg)
\end{align*}
where we denote
$c_{31} := \int_{\mathbb{S}^2} x_1^4\,\d x = 
\int_{\mathbb{S}^2} x_2^4\,\d x = \int_{\mathbb{S}^2} x_3^4\,\d x$ and
$c_{32} := \int_{\mathbb{S}^2}x_1^2x_3^2\d x 
=\int_{\mathbb{S}^2}x_2^2x_3^2\,\d x
=\int_{\mathbb{S}^2}x_1^2x_2^2\,\d x$.
Taking into account that $\lambda_3=-\lambda_1-\lambda_2$ we have:
\[
 \begin{split}
  &\lambda_1\lambda_2(\lambda_1+\lambda_2) + \lambda_1\lambda_3(\lambda_1+\lambda_3) 
  + \lambda_2\lambda_3(\lambda_2+\lambda_3) \\
  &\qquad\qquad\qquad = 3\lambda_1\lambda_2(\lambda_1+\lambda_3)
  = -(\lambda_1^3+\lambda_2^3+\lambda_3^3) = -\tr(Q^3).
 \end{split}
\]
 The proportionality constant $c_{31}- c_{32}$
can be computed explicitely, using spherical coordinates. We have
\begin{gather*}
  c_{31} = \int_{\mathbb{S}^2} x_1^4 \, \d x 
  = \left(\int_0^\pi\sin^5\theta\,\d\theta\right)
  \left(\int_0^{2\pi}\cos^4\phi\,\d\phi\right) 
  = \frac{16}{15}\cdot\frac{3}{4}\pi = \frac{4}{5}\pi, \\
  c_{32} = \int_{\mathbb{S}^2} x_1^2x_2^2 \, \d x 
  = \left(\int_0^\pi\sin^5\theta\,\d\theta\right)
  \left(\int_0^{2\pi}\cos^2\phi\sin^2\phi\,\d\phi\right) 
  = \frac{16}{15}\cdot\frac{1}{4}\pi = \frac{4}{15}\pi,
\end{gather*}

and~\eqref{tr3} follows.

\item Proof of~\eqref{tr4}. Similarly as in part $1$ we get:
\bea
&\int_{\mathbb{S}^2} (\nu\cdot Q^2\nu)^2 \,\d\nu
=\int_{\mathbb{S}^2}(\lambda_1^2 x_1^2+\lambda_2^2 x_2^2+\lambda_3^2 x_3^2)^2\,\d x\non\\
&\qquad\qquad =\int_{\mathbb{S}^2} \lambda_1^4 x_1^4+\lambda_2^4x_2^4+\lambda_3^4x_3^4+2(\lambda_1^2\lambda_2^2x_1^2x_2^2+\lambda_1^2\lambda_3^2x_1^2x_3^2
+\lambda_2^2\lambda_3^2x_2^2x_3^2)\d x\non\\
&\qquad\qquad =(\lambda_1^4+\lambda_2^4+\lambda_3^4)\underbrace{\int_{\mathbb{S}^2}x_1^4\,\d x}_{=c_{31}}+2(\lambda_1^2\lambda_2^2+\lambda_1^2\lambda_3^2+\lambda_2^2\lambda_3^2)\underbrace{\int_{\mathbb{S}^2} x_1^2x_2^2\d x}_{=c_{32}}
\eea
On the other hand, taking into account that $\lambda_1+\lambda_2+\lambda_3=0$ one can check through a straightforward calculation that 
$$
\textrm{tr}(Q^4)=\lambda_1^4+\lambda_2^4+\lambda_3^4=\frac{1}{2}(\lambda_1^2+\lambda_2^2+\lambda_3^2)^2=\frac{1}{2}(|Q|^2)^2
$$ 
which in particular implies:
$$
\lambda_1^4+\lambda_2^4+\lambda_3^4=2(\lambda_1^2\lambda_2^2+\lambda_1^2\lambda_3^2+\lambda_2^2\lambda_3^2)
$$
and thus, taking into account the previous calculations:
$$
\int_{\mathbb{S}^2} (\nu\cdot Q^2\nu)^2 \, \d\nu
=(\lambda_1^4+\lambda_2^4+\lambda_3^4)(c_{31}+c_{32})
=\frac{c_{31}+c_{32}}{2}|Q|^4.
$$

\item Proof of~\eqref{tr2bis}. The proof is completely 
analougous to that of~\eqref{tr4}. We have
\bea
\int_{\mathbb{S}^2} (\nu\cdot Q\nu)^2 \,\d\nu
&=\int_{\mathbb{S}^2}(\lambda_1 x_1^2+\lambda_2 x_2^2+\lambda_3 x_3^2)^2\,\d x\non\\
&=(\lambda_1^2+\lambda_2^2+\lambda_3^2)\underbrace{\int_{\mathbb{S}^2}x_1^4\,\d x}_{=c_{31}}+2(\lambda_1\lambda_2+\lambda_1\lambda_3+\lambda_2\lambda_3)\underbrace{\int_{\mathbb{S}^2} x_1^2x_2^2\d x}_{=c_{32}} \non
\eea
and, due to $\lambda_1+\lambda_2+\lambda_3=0$, we obtain
\[
\abs{Q}^2=\lambda_1^2+\lambda_2^2+\lambda_3^2
= - 2(\lambda_1\lambda_2+\lambda_1\lambda_3+\lambda_2\lambda_3),
\]
so that
$$
\int_{\mathbb{S}^2} (\nu\cdot Q\nu)^2 \, \d\nu
=(c_{31}-c_{32})|Q|^2.
$$
\end{enumerate}

\bigskip 
Finally we consider a function $f\colon\mcS_0\times\R^3\to\R$ satisfying the invariance condition~\eqref{invariance+}. Then, from Proposition~\ref{prop:surface_energy} we know that there exists $\tilde f:\R^5\to \R$ such that 
 \[
  f(Q, \, u) = \tilde{f}(\tr Q^2, \, \tr Q^3, \, |u|^2, \, u\cdot Qu, \, u\cdot Q^2u)
 \]
 for all~$Q\in\mcS_0$ and~$u\in\R^3$. 
 
 In order to prove the existence of the claimed $h$ function it suffices to show that 
 
\be  \int_{\mathbb{S}^2}f(Q,\nu)\,\d\nu= \int_{\mathbb{S}^2}f(RQR^{\mathsf{T}},\nu)\,\d\nu
\ee for any $R\in SO(3)$. To this end we use the function $\tilde f$ above and note that we have:

\bea
 \int_{\mathbb{S}^2}f(Q,\nu)\,\d\nu&=\int_{\mathbb{S}^2}\tilde{f}(\tr Q^2, \, \tr Q^3, \, 1, \, \nu\cdot Q\nu, \, \nu\cdot Q^2\nu )\d\nu\non\\
 &=\int_{\mathbb{S}^2}\tilde{f}(\tr Q^2, \, \tr Q^3, \, 1, \, R^{\mathsf{T}}\nu\cdot QR^{\mathsf{T}}\nu, \, R^{\mathsf{T}}\nu\cdot Q^2R^{\mathsf{T}}\nu)\d\nu\non\\
 &=\int_{\mathbb{S}^2}\tilde{f}(\tr Q^2, \, \tr Q^3, \, 1, \, \nu\cdot RQR^{\mathsf{T}}\nu, \, \nu\cdot RQ^2R^{\mathsf{T}}\nu)\d\nu\non\\
 &=\int_{\mathbb{S}^2}\tilde{f}(\tr (RQR^{\mathsf{T}})^2, \, \tr (RQR^{\mathsf{T}})^3, \, 1, \, \nu\cdot RQR^{\mathsf{T}}\nu, \, \nu\cdot (RQR^{\mathsf{T}})^2\nu)\d\nu\non\\
 &= \int_{\mathbb{S}^2}f(RQR^{\mathsf{T}},\nu)\,\d\nu
\eea
\end{proof}
\end{appendix}


\end{document}